\documentclass[smallextended]{svjour3}       
\usepackage{amssymb} 
\smartqed  
\usepackage{graphicx}

\newcommand{\nn}{n_0} 

\usepackage{enumitem}
\usepackage{newclude}
\usepackage{versions}
\usepackage{graphicx}
\usepackage{fancyvrb}

\usepackage[english]{babel}
\usepackage{amsfonts}
\usepackage{amstext}
\usepackage{amsmath}
\usepackage{amssymb,dsfont,mathtools}
\usepackage{bbm} 
\usepackage{epsfig, subfigure}
\usepackage[square,numbers]{natbib}

\usepackage{color}

\usepackage[colorlinks,citecolor=blue,urlcolor=blue]{hyperref}
\RequirePackage{hypernat}

\usepackage{url}

\usepackage{multirow}
\usepackage{booktabs}

\usepackage{xparse} 
\usepackage{etoolbox} 

\usepackage[usenames,dvipsnames]{xcolor}

\NewDocumentCommand\Nb{O{}}{
  \ifstrempty{#1}{
    N_{[\, ]}
  }{
    N_{[\, ]}\left( {#1} \right)
  }
}

\NewDocumentCommand\Nc{O{}}{
  \ifstrempty{#1}{
    N
  }{
    N\left( {#1} \right)
  }
}

\NewDocumentCommand\cC{O{}}{
  \ifstrempty{#1}{
    \mc C
  }{
    \mc C \left( {#1} \right)
  }
}

\NewDocumentCommand\cCck{O{}}{
  \ifstrempty{#1}{
    \mc{C}_{c,k}
  }{
    \mc{C}_{#1}
  }
}

\NewDocumentCommand\cCciki{O{}}{
  \ifstrempty{#1}{
    \mc{C}^0_{c}
  }{
    \mc{C}_{#1}
  }
}

\newcommand{\Sxy}{{\widehat{S}^0}}
\newcommand{\HR}{{H_R}} 
\newcommand{\HL}{{H_L}}
\newcommand{\YR}{{Y_R}}
\newcommand{\YL}{{Y_L}}
\newcommand{\XL}{{X_L}}
\newcommand{\XR}{{X_R}}
\newcommand{\FL}{{F_L}}
\newcommand{\FR}{{F_R}}
\newcommand{\cCn}{{\cC}_n}
\newcommand{\cCxy}{{\cC}^0} 
\newcommand{\cCxyn}{{\cC}^0_{n}} 
\newcommand{\rn}{\widehat{r}_n}

\newcommand{\rxy}{\widehat{r}^0}
\newcommand{\rxyc}{\widehat{r}^0_c}
\renewcommand{\r}{\widehat{r}}
\newcommand{\rnxy}{\widehat{r}_{n}^0}

\newcommand{\urn}{\underline{\widehat{r}}_n}
\newcommand{\urnxy}{\underline{\widehat{r}}_n^0}
\newcommand{\uYn}{\underline{Y}_n}
\newcommand{\Ix}{I}

\NewDocumentCommand{\rni}{O{n,i}}{
\widehat{r}_{#1}
}
\NewDocumentCommand{\rnixy}{O{n,i}}{
\widehat{r}^0_{#1}
}
\NewDocumentCommand{\Rnk}{O{n,k}}{
\widehat{R}_{#1}
}
\NewDocumentCommand{\RnRk}{O{n,k}}{
\widehat{R}^R_{#1}
}
\NewDocumentCommand{\RnLk}{O{n,k}}{
\widehat{R}^L_{#1}
}
\NewDocumentCommand{\Snk}{O{n,k}}{
S_{#1}
}
\NewDocumentCommand{\SnRk}{O{n,k}}{
S^R_{#1}
}
\NewDocumentCommand{\SnLk}{O{n,k}}{
S^L_{#1}
}

\NewDocumentCommand{\Yni}{O{n,i}}{
Y_{#1}
}
\NewDocumentCommand{\tYni}{O{n,i}}{
\widetilde{Y}_{#1}
}

\NewDocumentCommand{\Wni}{O{n,i}}{
Y_{#1}
}
\NewDocumentCommand{\xni}{O{n,i}}{
x_{#1}
}
\NewDocumentCommand{\txni}{O{n,i}}{
\widetilde{x}_{#1}
}
\NewDocumentCommand{\zni}{O{n,i}}{
x_{#1}
}

\global\long\def\vvna{\widehat{\vp}_{n}^{0}}

\global\long\def\vvna{\widehat{\vp}_{n}^{0}}

\global\long\def\vvn{\widehat{\vp}_{n}}

\global\long\def\TauL{\tau_{L}} 
\global\long\def\TauR{\tau_{R}} 
\global\long\def\tauplusa{\tau_{+}}

\global\long\def\YnRloc{Y_{n,R}}
\global\long\def\HnRt{\tilde{H}_{n,R}}
\global\long\def\HnR{{H}_{n,R}}
\global\long\def\RRnRt{\tilde{\RR}_{n,R}}
\global\long\def\RRnR{{\RR}_{n,R}}
\global\long\def\TaunR{\tau_{n,R}} 
\global\long\def\TaunL{\tau_{n,L}} 

\global\long\def\inv#1{\frac{1}{#1}}

\DeclareMathOperator*{\Var}{Var}

\DeclareMathOperator{\eval}{eval}
\DeclareMathOperator{\ext}{ext}
\DeclareMathOperator{\dom}{dom}
\DeclareMathOperator{\cone}{cone}

\DeclareMathOperator*{\argmin}{argmin}
\DeclareMathOperator*{\argmax}{argmax}

\newcommand{\lp}{\left(} 
  \newcommand{\rp}{\right)}
\newcommand{\lb}{\left\{} 
  \newcommand{\rb}{\right\}}
\newcommand{\ls}{\left[} 
  \newcommand{\rs}{\right]}
\newcommand{\la}{\left\langle}
\newcommand{\ra}{\right\rangle}
\newcommand{\ip}[1]{\la #1 \ra}

\newcommand{\vp}{\varphi}

\newcommand{\ve}[1]{\boldsymbol{#1}}

\newcommand{\DD}{{\mathbb D}}

\newcommand{\RR}{\mathbb{R}}
\renewcommand{\SS}{\mathbb{S}} 

\newcommand{\one}{\mathbbm{1}}
\newcommand{\mc}[1]{\mathcal {#1}}

\newcommand{\iid}{\stackrel{\text{iid}}{\sim}}

\newtheorem{assumption}{Assumption}

\newenvironment{altquote}
               {\list{}{\rightmargin\leftmargin}%
                \item\relax\normalsize\ignorespaces}
               {\unskip\unskip\endlist}

\excludeversion{mylongform}

\excludeversion{myscribbles}

\excludeversion{myextra}

\excludeversion{supplement-file}

\excludeversion{main-file}

\includeversion{all-in-one-file}
\newcommand{\mydoctitle}{Concave regression: value-constrained estimation and likelihood ratio-based inference}

\begin{document}


\begin{main-file}
  \title{\mydoctitle \thanks{Supported in part by NSF Grant DMS-1712664}
}
\end{main-file}

\begin{all-in-one-file}
  \title{\mydoctitle \thanks{Supported in part by NSF Grant DMS-1712664}} 
\end{all-in-one-file}


\titlerunning{Concave regression: constrained estimation and LRTs}        

\author{Charles R. Doss}

\authorrunning{C.\ R.\ Doss} 



\institute{
  University of Minnesota \\
  School of Statistics  \\
  313 Ford Hall, \\
  224 Church Street SE, \\
  Minneapolis, MN 55455 \\
  U.S.A. \\
  \email{cdoss@stat.umn.edu}           
}

\date{Received: date / Accepted: date}

\maketitle

\begin{abstract}

We propose a likelihood ratio statistic %
for forming hypothesis tests and confidence intervals for a nonparametrically
estimated univariate regression function, based on the shape restriction of
concavity 
(alternatively, convexity).
Dealing with the likelihood ratio statistic requires studying an estimator
satisfying a null hypothesis,
that is, studying a concave least-squares estimator satisfying a further
equality constraint.
We study this null hypothesis least-squares estimator (NLSE) here,
and use it to study our likelihood ratio statistic.
The
NLSE %
is the solution to a convex
program, and we find a set of inequality and equality constraints that
characterize the solution.
We also study a corresponding limiting version of the convex program based on
observing a Brownian motion with drift.
The solution to the limit problem is a stochastic process.
We study the optimality conditions for the solution to the limit problem
and find that they match those we derived for the solution to the finite
sample problem.
This allows us to show the limit stochastic process yields the limit
distribution of the (finite sample) NLSE.
We conjecture that the
likelihood ratio statistic is asymptotically pivotal, meaning that it has a
limit distribution with no nuisance parameters to be estimated, which makes
it a very effective tool for this difficult inference problem.
We provide a partial proof of this conjecture, and
we  also  provide simulation evidence strongly supporting this conjecture.
\end{abstract}











\tableofcontents



%
%

\section{Introduction}

In nonparametric density, regression, or other function %
estimation, forming hypothesis tests and confidence intervals is important but often challenging.  For nonparametric estimators to be effective, they are generally tuned so as to balance their bias and variance (perhaps asymptotically).  However, having non-negligible asymptotic bias is problematic for doing inference, since the bias must then be assessed to do honest and efficient inference.  One approach is to ignore the bias (e.g., Chapter~5.7 of \cite{Wasserman:2006uf}), although this is clearly problematic.  Often the bootstrap %
\cite{Efron:1979ha} %
can be used for inference in complicated problems, but it is frequently a poor estimate of bias and so requires corrections or modifications.
Such corrections have been implemented in a large variety of cases.
For instance in forming confidence intervals for a density function, one approach is to undersmooth a kernel density estimator and then use the bootstrap
\cite{Hall:1992eb}.  However the undersmoothed estimator used for the confidence interval is then different from that which would be optimal for pure estimation, and requires stronger smoothness assumptions than would be required for just estimation.
Importantly, the inference is still dependent on a tuning parameter (the bandwidth), whose optimal selection can be challenging, can lead to different inferences for different users, and can add another layer of computational burden.

These issues motivate an alternative approach to
nonparametric function estimation and inference, which relies on assumptions based on shape constraints and which often does not suffer from
the above problems. Here we consider the regression setup,
\begin{equation}
  \label{eq:1}
  \Wni  = r_0(\xni) + \epsilon_{n,i},
  \qquad
  i=1,\ldots, n,
\end{equation}
where $\Wni \in \RR$, we %
assume that the univariate
predictor variables $\xni$ are fixed,
and $\epsilon_{n,i}$ are independent and identically distributed
(i.i.d.) with mean $0$,
and $Ee^{t \epsilon_{n,i}^2} < \infty$ for some $t > 0$.
We assume that the target of estimation, $r_0 \colon \RR \to \RR$,  is concave.
(Concave regression is equivalent to convex regression by
taking $-\Wni$ as our responses; we will sometimes use ``concave/convex regression'' to mean either concave regression or convex regression since they are equivalent.)
As will be discussed in greater detail below, concave/convex regression estimators are solutions to convex programs, and so they have very different properties than many other nonparametric regression estimators such as kernel-based ones.
Concave/convex regression estimation arises in a truly vast number of settings.
It seems to have originally arisen in the econometrics literature \cite{Hildreth:1954tz}.
As noted by \cite{Hildreth:1954tz},
in %
classical %
economic theory
\begin{altquote}\label{loc:shape-constrained-functions-economics}
  utility functions are usually assumed to be concave; marginal
  utility is often assumed to be convex; and functions representing
  productivity, supply, and demand curves are often assumed to be either
  concave or convex.
\end{altquote}
(The example worked through by \cite{Hildreth:1954tz}  is on production function estimation.)
Related examples in finance also exhibit convexity restrictions (\cite{AtSahalia:2003gy} study stock option pricing).
Concavity/convexity also arises in operations research, where the concavity/convexity often arises theoretically,
and then conveniently makes optimization of the estimated function  very efficient
(\cite{Topaloglu:2003eu},
\cite{Lim:2010hm},
\cite{Toriello:2010gl},
\cite{Hannah:2013vf}).
See \cite{Monti:2005kh} and references therein
for  further examples of uses of concavity restrictions.
There have been a variety of works that have considered concave/convex regression, in the literatures of different fields, with most of the focus being on estimation
(\cite{Hanson:1976to},
\cite{Birke:2007ep},
\cite{Allon:2007bo},
\cite{Kuosmanen:2008ks},
\cite{Seijo:2011ko},
\cite{Lim:2012do},
\cite{Pflug:2013kf}).
\cite{Meyer:2008ba},
\cite{Wang:2011iu},
and
\cite{Meyer:2012el} consider spline-based approaches to concave regression estimation, and study testing the hypothesis of linearity against concavity/convexity, and testing the hypothesis of concavity/convexity against a general smooth alternative.
\cite{Mammen:1991ey} finds the rates of convergence of the univariate LSE, and
\cite{Groeneboom:2001fp,Groeneboom:2001jo} find its limit distribution.
Algorithms for computing estimators in concave regression settings (sometimes combined with other constraints) have been studied
by
\cite{Hudson:1969cs},
\cite{Dent:1973bp},
\cite{Wu:1982tw},
\cite{Dykstra:1983vd},
\cite{dasFraser:1989di},
\cite{Meyer:2008ba},
and
\cite{Meyer:2012el}.
In \cite{Cai:2013dua}, the authors study upper and lower bounds for the lengths of confidence intervals (with a fixed coverage probability) for concave regression, but we do not know of any practical implementation for the intervals they study.
In the Gaussian white noise model, \cite{Dumbgen:2003ep} studies multiscale confidence bands (rather than pointwise intervals) for a concave function.  (Confidence bands can of course be used for pointwise confidence intervals but will be unnecessarily long.)

The model for $r_0$, based only on the assumption that $r_0$ is concave, is nonparametric and infinite dimensional. However, it is still possible
to  estimate $r_0$ directly via least-squares, as in finite dimensional problems.  We let
\begin{equation}
  \label{eq:defn:LS-objective}
  \rn := \argmin_{r} \phi_n(r) :=  \argmin_{r} \inv{2} \sum_{i=1}^n ( \Wni - r(\xni))^2
\end{equation}
where the argmin is taken over all concave functions $r \colon \RR \to \RR$.
Perhaps surprisingly,
minimizing the least-squares objective function over the class of all functions constrained only to be concave admits a solution that is uniquely specified at the data points.
It is possible for the solution to be not uniquely specified at some other points,
so we take $\rn$ to be piecewise linear between the $\xni$'s \cite{Hildreth:1954tz}.
The limit distribution of the estimator at a fixed point $x$ has been obtained (under a second derivative assumption and uniformity conditions on the design of the $\xni$'s) by
\cite{Groeneboom:2001jo},
who show that
\begin{equation}
  \label{eq:GJW-asymptotics}
  d(r_0) n^{2/5} (\rn(x_0) -  r_0(x_0)) \to_d U
\end{equation}
where $U$ %
is a universal limit distribution (meaning it does not depend on $r_0$), and $d( r_0) := (24 / \sigma^4 | r_0''(x_0)|)^{1/5}$, where $\sigma^2 = \Var(\epsilon_{n,i})$.   (In fact, $U \equiv \r_{1,1}(0)$ where $\r_{1,1}$ is described below in Theorem~\ref{thm:GJW-process-uniqueness}.)  We use $g' \equiv g^{(1)}$, $g'' \equiv g^{(2)}$, and $g^{(i)}$ to refer to the first, second, and $i$th derivatives of an appropriately differentiable
function $g$.

One might attempt to directly use
the limit result \eqref{eq:GJW-asymptotics}
as the basis for inference about $r_0(x_0)$.  However, the limit distribution depends on $r_0''(x_0)$, and so using \eqref{eq:GJW-asymptotics} requires somehow estimating $r_0''(x_0)$, which leads to many of the problems described in the first paragraph of this paper.  We avoid this, rather pursuing a hypothesis test approach based on a likelihood ratio statistic (LRS), %
and using that to develop a confidence interval.
(Here `likelihood ratio' is a slight abuse of terminology, since it will be a likelihood ratio only if the $\epsilon_{n,i}$ are Gaussian, which  we do not assume; the LRS could alternatively be referred to as a residual-sum-of-squares statistic.)
We will consider the hypothesis test
\begin{equation}
  \label{eq:LR-hypothesis-test}
  H_0: r_0(x_0) = y_0
  \qquad \mbox{ against } \qquad
  H_1:  r_0(x_0) \ne y_0
\end{equation}
for $(x_0,y_0) \in \RR^2$ fixed.  To form confidence intervals, we will invert the hypothesis test: assume we reject $H_0$ when $2 \log \lambda_n(y_0) > d_\alpha$ for a statistic $2 \log \lambda_n(y_0)$ (to be discussed shortly) and some critical value $d_\alpha$, $\alpha \in (0,1)$.  Then the corresponding confidence interval is
\begin{equation}
  \label{eq:defn:CI}
  \lb y : 2 \log \lambda_n(y) \le d_\alpha \rb.
\end{equation}
(Since $y$ is univariate, the confidence interval can be computed by computing the test on a grid of $y$ values.)

The statistic $2 \log \lambda_n(y_0)$ which we will study is based on a ratio statistic $\lambda_n(y_0)$ which depends on a
`null hypothesis statistic' and on an `alternative hypothesis statistic.'
The null hypothesis statistic will depend on a least-squares estimator (LSE) of a concave regression function, $\rnxy$,  that is further constrained so that $\rnxy(x_0)=y_0$ where  $y_0$ is fixed:
$\rnxy := \argmin_r \sum_{i=1}^n (\Wni - r(\xni))^2$ where the argmin is over concave functions $r$ satisfying $r(x_0) = y_0$.
We refer to this estimator as the `null hypothesis least-squares estimator' (NLSE).
The `alternative hypothesis statistic'
depends on $\rn$,
which we thus refer to as the `alternative hypothesis least-squares estimator' (ALSE).
With these two estimators in hand, we define our statistic by
\begin{equation}
  \label{eq:defn:TLLR}
  2 \log \lambda_n \equiv 2 \log \lambda_n(y_0) := 2(\phi_n(\rnxy) - \phi_n(\rn))
\end{equation}
with $\phi_n$ defined in \eqref{eq:defn:LS-objective}.
One of the major benefits to using LRS's is that their limit distribution often does not depend on nuisance parameters.  In regular parametric problems, two times the log of the
LRS is asymptotically $\chi^2_k$, where $k$ is the reduction in parameter dimension in going from the alternative hypothesis to the null hypothesis.  Notably, this chi-squared distribution is universal, meaning it is the same limit distribution regardless of what underlying parameter is the true one, so no nuisance parameters need to be estimated to perform inference, which can make inference more simple and more efficient.
(In this case, one says that log likelihood ratios are (asymptotically) pivotal, or that they satisfy the Wilks phenomenon.)

In our shape-constrained setting, LRS's can be challenging to analyze theoretically.  However, such analysis has been successful in some cases.
\cite{Banerjee:2001jy} and \cite{Groeneboom:2015ew} study LRS's based on monotonicity shape constraints, and \cite{Doss:2016ux,Doss:2016vq}  consider an LRS based on the concavity shape constraint.
The estimators underlying both of these tests are maximum likelihood estimators, and they do not require any tuning parameter selection.  The LRS's were shown to have asymptotic distributions that are universal, not depending at all on the unknown true function, so do not require any additional procedures for their use for inference.  Also, the assumptions needed for the LRS asymptotics to hold are the same as those for estimation, rather than stronger ones as in some other nonparametric settings.

These positive results motivate interest in using the statistic
$2 \log \lambda_n$ of \eqref{eq:defn:TLLR} for testing and forming confidence
intervals for $r_0(x_0)$, and suggest that it may have a limit distribution
that is universal and free of nuisance parameters.
This would allow us to avoid
the difficult estimation of $r_0''(x_0)$ and resulting tuning parameter selection problem
that would be required if we rely on \eqref{eq:GJW-asymptotics} to do inference.
We make the
following conjecture.  To state the conjecture we need some assumptions on
the design variables (and on $x_0$); the assumptions (Assumption~\ref{assm:design-density-1}
and \ref{assm:local-uniform-design-2}) are stated and discussed in
Section~\ref{subsec:LRS-convergence}.

\medskip
\begin{conjecture}
  \label{conj:Wilks-phenomenon}
  Assume the regression model \eqref{eq:1} holds where
  $Ee^{t \epsilon_{n,i}^2} < \infty$ for some $t > 0$.  Assume ${r}_0$
  is concave, ${r}_0(x_0) = y_0$, ${r}_0$ is twice
  continuously differentiable in a neighborhood of $x_0$, and
  ${r}_0''(x_0)<0$.  Let
  Assumption~\ref{assm:design-density-1} and
  \ref{assm:local-uniform-design-2} hold.  Then, with $2 \log \lambda_n(y_0)$
  defined in \eqref{eq:defn:TLLR},
  \begin{equation}
    \label{eq:Wilks-phenomenon}
    2 \log \lambda_n(y_0) \to_d  \sigma^2 \DD,
  \end{equation}
  where $\DD$ is a universal random variable (not depending on ${r}_0$ or the distribution of $\epsilon_{n,i}$).
\end{conjecture}
\medskip

\noindent
A partial proof of the conjecture is given in 
 Subsection~\ref{subsec:LRS-convergence}. See 
Theorem~\ref{thm:conjecture-theorem} there.
The form of the random variable $\DD$ is given below in \eqref{eq:wilks-limit-3}.
Some discussion of the assumptions is given in remarks after Theorem~\ref{thm:constrained-regression-asymptotics-sec1}.

A theorem analogous to Conjecture~\ref{conj:Wilks-phenomenon} was proved by \cite{Banerjee:2001jy} (see also \cite{Banerjee:2000we}) in the context of the current status data model of survival analysis,
by \cite{MR2341693} in the context of monotone response models,
and by \cite{Groeneboom:2015ew}   in
the context of monotone density estimation.  Those models are based on the shape restriction of monotonicity.  In the context of a shape restriction based on concavity, \cite{Doss:2016ck,Doss:2016ux,Doss:2016vq} show a theorem analogous to Conjecture~\ref{conj:Wilks-phenomenon} for an LRS for the mode of a log-concave density.  The likelihood ratio in the latter problem, based on a concavity assumption, involves remainder terms which are asymptotically negligible but are quite challenging to theoretically analyze.  In the current status problem there are no such remainder terms, and in the monotone density problem they can be analyzed using the so-called min-max formula (see e.g., Lemma~3.2 of \cite{Groeneboom:2015ew}), which does not have an analog for concavity-based problems.  Thus it is quite difficult in general to analyze LRS's in concavity-based problems, and so proving Conjecture~\ref{conj:Wilks-phenomenon} in full is a large undertaking  beyond the scope of the present paper.
To study the asymptotics of $2 \log \lambda_n$ and prove Conjecture~\ref{conj:Wilks-phenomenon}, one needs to study the asymptotics of the constrained estimator $\rnxy$.  Since $\rnxy$ is the solution to a strictly convex program, there are
optimality conditions that characterize it (i.e., Karush-Kuhn-Tucker type conditions).
One key component in developing the asymptotics of $\rnxy$
is to understand the conditions that characterize $\rnxy$, which we do
in Theorem~\ref{prop:rnxy-characterization}.
We also study a corresponding limit version of the problem, which is to find the constrained concave least-squares estimator based on observing a Brownian motion with drift (i.e., observing the solution to a stochastic differential equation).
We find conditions that characterize the solution to this limit problem (the limit LSE) in Theorem~\ref{thm:characterization-compact} (on a compact domain)
and Theorem~\ref{thm:characterization-RR-sec1}  (on all of $\RR$)
and we see that the conditions are analogous to those in the finite sample case.
(Theorem~\ref{thm:characterization-compact} is used to prove Theorem~\ref{thm:characterization-RR-sec1}.)
Showing that the convex program optimality conditions are the same for the finite sample estimator
and for the limit process  is a crucial step in showing the limit process is indeed the limit distribution of the finite sample estimator.
Finding the characterizing conditions, particularly in the limit problem, seems to be somewhat more challenging for the constrained problems than for the unconstrained ones.
The process arising in Theorem~\ref{thm:characterization-RR-sec1} is used in Theorem~\ref{thm:constrained-regression-asymptotics-sec1}, which gives the limit distribution of $\rnxy(x_0)$.  Finally, in Subsection~\ref{subsec:LRS-convergence}, we use Theorem~\ref{thm:constrained-regression-asymptotics-sec1} to give a partial proof of of Conjecture~\ref{conj:Wilks-phenomenon}.  Specifically, in Theorem~\ref{thm:conjecture-theorem} we show that under an assumption on a certain remainder term, the conjectured limit statement \eqref{eq:Wilks-phenomenon} holds.

We further describe the structure of this paper, as follows.
In Section~\ref{sec:regression-LSEs-finitesample} we consider the regression model and study some basic properties of the (finite sample) NLSE and ALSE, which includes presenting
Theorem~\ref{prop:rnxy-characterization}.  %
In
Section~\ref{sec:regression-LSEs-limit}
we study the limiting version of the
problem
and present
Theorem~\ref{thm:characterization-compact}
and Theorem~\ref{thm:characterization-RR-sec1}.
In
Section~\ref{sec:asymptotics}
we present
Theorem~\ref{thm:constrained-regression-asymptotics-sec1}.
In
Subsection~\ref{subsec:LRS-convergence}
we present a partial
proof of Conjecture~\ref{conj:Wilks-phenomenon}.
In Section~\ref{sec:simulations},
we provide simulations giving strong evidence in favor of
Conjecture~\ref{conj:Wilks-phenomenon}, and showing that the corresponding test and confidence interval %
have good finite sample performance.
Section~\ref{sec:conclusion} has some concluding remarks and discussion of related problems.  Appendix~\ref{sec:appendix} has results we include for completeness and technical formulas.

\begin{myextra}
  \section{A likelihood ratio test for a log-concave density's mode}
  \label{sec:LRT-LC-mode}

  \input{p_sec-LR-mode.tex}
\end{myextra}

\section{Finite sample constrained concave regression}
\label{sec:regression-LSEs-finitesample}
We begin with the
regression setup
$  \tYni  = \tilde r_0(\txni) + \epsilon_{n,i},$
$i=1,\ldots, n.$
We assume  that $\epsilon_{n,i}$ are %
i.i.d.\ with mean $0$,
$Ee^{t \epsilon_{n,i}^2} < \infty$ for some $t > 0$, we assume
$\lb \txni \rb$ are fixed and
without loss of generality we assume that
$\txni[n,1] < \txni[n,2]  < \cdots < \txni[n,n]$.
Our model assumption is that $\tilde{r}_0: \RR \to \RR$ is a concave function.
Our interest is in using $2 \log \lambda_n$
from \eqref{eq:defn:TLLR}
to test for the value of $\tilde{r}_0$ at a fixed point $x_0$,  and also in inverting those tests to form corresponding confidence intervals.
Thus, we will study the constrained concave regression
problem, where at a fixed point $x_0 \in \RR$ we assume $\tilde{r}_0(x_0) =
y_0$ for a fixed value $y_0$.

Let
\begin{equation}
  \label{UnconstrainedConcaveClass}
  {\cal C} := \{ \vp : \ \RR \rightarrow [-\infty, \infty) \ | \ \vp
  \ \   \mbox{is concave, closed, and proper} \}
\end{equation}
Here $\vp$ is proper if $\vp(x) < \infty$ for all $x$ and $\vp(x) > -\infty$ for some $x$
and $\vp$ is closed if it is upper semi-continuous (as in \cite{Rockafellar:1970wy},
pages 24 and 50).
We follow the convention that a concave function $\vp$ is defined on all of $\RR$ by
assigning
$\vp$ the value $-\infty$ off its effective domain
$\dom(\vp):= \{ x \ : \ \vp (x) > -\infty \}$
(as in \cite{Rockafellar:1970wy},  page 40).
For fixed $ (x_0,y_0) \in \RR^2$, let
$\tilde \cCxy := \{ r \in \cC \ : \ r(x_0)=y_0 \}.$
We consider  estimation of $\tilde r_0$
via minimization of the objective function
$  r\mapsto  \inv{2} \sum_{i=1}^n ( \tYni  - r(\txni) )^2.$
The constrained  %
LSE %
is the minimum of the above objective function
over $\tilde \cCxy$; however, $\tilde \cCxy$ is not a convex cone.
Thus, to proceed further, we now introduce an augmented or auxiliary data
set.  We will (a) translate the original data set so that the corresponding
set of possible regression functions forms a convex cone, and (b) potentially
augment the $\txni$ by $x_0$.
In addition, in (a), without loss of generality, we will translate the data so that the true regression function may be assumed to satisfy $r_0'(x_0)=0$.
We define the auxiliary data set $\{ (\zni, \Wni) \}_{i=1}^{\nn}$, where $\nn$ will be either $n$ or $n+1$, as follows.
\begin{enumerate}
\item If $x_0$ is equal to one of the data points, say $\txni[n,k^0]=x_0$
  where $1 \le k^0 \le n$, then let $\xni := \txni $ and let
  $\nn := n$.  %
  Let $\Wni := \tYni - y_0 - \tilde{r}_0'(x_0)(\txni-x_0)$ for $i = 1, \ldots, n$.
\item If $x_0$ is not equal to any data point, then let $1 \le k^0 \le n+1$
  be such that $\txni[n,k^0-1] < x_0 < \txni[n,k^0]$, where we let
  $\txni[n,0]=-\infty$ and $\txni[n,n+1]=\infty$ here. Then for $i=1,\ldots,
  k^0-1$ let $\zni :=\txni$ and
  $\Wni := \tYni - y_0- \tilde{r}_0'(x_0)(\txni-x_0)$,
  let $\zni[n,k^0] := x_0$, and
  for $i=k^0+1, \ldots, n+1 =: \nn$, let $\zni := \txni[n,i-1]$ and $\Wni :=
  \tYni[n,i-1] - y_0 - \tilde{r}_0'(x_0)(\txni[n,i-1]-x_0)$.
  Define $\Wni[n,k^0] := 0$.  
\end{enumerate}
Thus the size of the augmented data set, $\nn$, is either $n$ or $n+1$.
In either case, define $\Ix$ to be the subset of $\lb 1, \ldots, \nn \rb$
corresponding to data indices, so $\Ix$ has cardinality $n$ and may or may
not include $k^0$.
Thus, with these definitions $\Wni$ and $\zni$ satisfy the regression relationship
\begin{equation}
  \label{eq:regression-augmented}
  \Wni = r_0(\zni) + \epsilon_{n,i} \quad \mbox{ for } \quad i \in \Ix,
\end{equation}
where $r_0 \in \cCxy := \{ r \in \cC \ | \ r(x_0)=0  \}.$
We thus consider the objective function\footnote{
  Note that \eqref{eq:regression-augmented} and \eqref{eq:8} are potentially different from
  \eqref{eq:1}
  and \eqref{eq:defn:LS-objective}
  in the introduction, but only by a minor indexing modification.}
\begin{equation}
  \label{eq:8}
  \phi_n(r) = \inv{2} \sum_{i
    \in \Ix} \lp \Wni - r(\zni) \rp^2.
\end{equation}
A
priori, $\argmin_{r \in \cC} \phi_n(r)$ is  uniquely specified only at the data points $\zni$ for $i \in \Ix$,
and $\argmin_{r \in \cCxy} \phi_n(r)$ is  only  uniquely specified at the  data points $\zni$ for $i=1, \ldots, \nn$; thus we choose to restrict attention to solutions that are affine between the $\zni$.  (The actual solutions will be uniquely specified on most of their domain in practice, because they will be piecewise linear with relatively few knot points.)
Restricting attention to piecewise affine
solutions is the standard approach, and the choice does not affect the asymptotic results, see e.g.\
\cite{Seijo:2011ko}.
For a
concave function that is piecewise linear, we can identify the function with
its values at its bend points, so we define corresponding subsets of $\RR^n$
and $\RR^{\nn}$ by
\begin{equation}
  \label{eq:defn:cCn}
  \cCn :=
  \{ \lp r(\txni[n,1]), \ldots, r(\txni[n,n]) \rp : r \in \cC \}
  \mbox{ and }
  \cCxyn :=
  \{ \lp r_0(\zni[n,1]), \ldots, r_0(\zni[n,\nn]) \rp : r_0 \in \cCxy \} .
\end{equation}
For a function $r$
we let $\eval_n r := \lp  r(\txni[n,1]), \ldots,  r(\txni[n,n]) \rp$ and
$\eval_{\nn} r:= (r(\zni[n,1]), \ldots, r(\zni[n,\nn]))$.
For $r_n \in \cCn$,
define the linear extrapolation $\ext(r_n) \in \RR^{\nn}$ %
by $\ext(r_n) := \eval_{n_0}(r)$, where $r$ is the function giving the linear interpolation of $r_n$.
Then, slightly abusing notation (by giving $\phi_n$ a $\RR^{\nn}$-vector argument rather than a function), we define the estimator vectors $\urn$ and $\urnxy$ by
\begin{equation}
  \label{eq:defn:UC-VC-LSEs}
  \urn \in \argmin_{r \in \cCn} \phi_n(\ext r)
  \quad \mbox{ and } \quad
  \urnxy \in \argmin_{r \in \cCxyn} \phi_n(r),
\end{equation}
and let $\rn$ and $\rnxy$ be the piecewise linear interpolation of $\urn$ and $\urnxy$ on $[\zni[n,1], \zni[n,\nn]]$.
We let $\rni := \rn(\txni)$, $i \in \Ix$, and
$\rnixy:= \rnxy(\zni)$, $i=1,\ldots, \nn$.

\begin{proposition}
  \label{prop:rnxy-existence-uniqueness}
  The estimators $\urn$ and $\urnxy$ exist and are unique (for any $n \ge 1$ or $\nn
  \ge 1$, respectively).
\end{proposition}
\begin{proof}
  Both statements follow from writing the optimization as a quadratic program with linear inequality constraints giving the concavity restriction and one (linear) equality constraint corresponding to $r(x_0)=0$ for the constrained estimator.
  \qed
\end{proof}

\noindent The estimators $\rn$ and $\rnxy$ can be seen as projections of the data onto the
convex cones $\cCn$ and $\cCxyn$, and so we begin by studying these cones.
A convex subset $\cal K$ of a (possibly infinite dimensional) real vector space is a convex cone
if  $x \in \cal K$ implies $\lambda x \in {\cal K}$ if $\lambda \ge 0$.
We say a convex cone $\cal K$ is (finitely) generated by or is spanned by  a set of
elements $k_1, \ldots, k_m \in \cal K$ if any $k \in {\cal K}$ can be written
as $k = \sum_{i=1}^m \lambda_i k_i$ for some $\lambda_i \ge 0$.
Define $(y)_- := \min(y, 0)$ for $y \in \RR$.

\begin{proposition}
  \label{prop:concave-generators}
  \begin{enumerate}%
  \item   \label{item:cCn-generators}  A generating set for $\cCn$ is given by $\pm \eval_n 1, \pm \eval_n x, $ and
    $\eval_n(\txni - x)_-$ for $i=2,\ldots, n-1$.
  \item \label{item:cCxyn-generators} A  generating set for $\cCxyn$ is given by
    $\pm \eval_{\nn}   (x-  x_0)$,
    $\eval_{\nn} (x-\zni)_-$ for $i=2,\ldots,k^0$,
    and $\eval_{\nn}(\zni-x)_- $ for $i=k^0+1,\ldots, {\nn}-1$.
  \end{enumerate}
  \label{prop:generators-cC}
\end{proposition}
\begin{proof}
  First we show \ref{item:cCn-generators}. %
  Consider the subset of $\cC$ that is piecewise affine with kinks only possible at $\txni$, and where we restrict attention to $[\txni[n,1], \txni[n,n]]$.  Then for $x \in [\txni[n,1], \txni[n,n]]$ we can write
  \begin{equation*}
    r(x) = b + w_1 (x- \txni[n,1]) + w_2( \txni[n,2] -x)_-
    + \cdots + w_{n-1} (\txni[n,n-1]- x)_-,
  \end{equation*}
  where $b, w_1 \in \RR$,
  since $r$ is piecewise affine,
  and $w_i \ge 0$ for $i = 2, \ldots, n-1$ since $r$ is concave.
  Thus $\pm 1$, $\pm x$, and $(\txni - x)_-$
  for $i=2, \ldots, n-1$
  generate the cone of piecewise affine functions on $[\txni[n,1], \txni[n,n]]$, and
  applying $\eval_n$ to these functions yields a generating set for $\cCn$.

  Now we show \ref{item:cCxyn-generators}. %
  Any $r$ that is piecewise affine with possible kinks at the $\zni$ can be
  written as
  \begin{equation}
    \label{eq:5}
    r(x) = b + w_1 (x- \zni[n,1]) + w_2(\zni[n,2] -x)_-
    + \cdots + w_{n-1} (\xni[n,{\nn}-1]- x)_-,
  \end{equation}
  where $b, w_1 \in \RR$, and $w_i \ge 0$ for $i = 2, \ldots, {\nn}-1$.
  Since $(\zni-x)_- = (x -  \zni)_- - (x-\zni)$, we can rewrite
  \eqref{eq:5}  as
  \begin{equation}
    \label{eq:6}
    \begin{split}
      \MoveEqLeft \tilde b + \tilde w_1 (x-x_0)
      + \tilde w_2(x - \xni[n,2])_- + \cdots + \tilde w_{k^0} (x- \xni[n,k^0])_- \\
      & \quad + w_{k^0+1} ( \xni[n,k^0+1] - x)_- + \cdots + w_{{\nn}-1} (\xni[n,{\nn}-1] - x)_-
    \end{split}
  \end{equation}
  where $\tilde b,  \tilde w_1 \in \RR$, and $w_i, \tilde w_i \ge 0$.
  Thus enforcing $r(x_0)=0$ amounts precisely to requiring $\tilde b = 0$ in
  \eqref{eq:6}.  Thus
  $\pm (x-x_0)$, $(x-\zni)_-$ for $i=2,\ldots,k^0$, and $(\zni-x)_-$ for $i=k^0+1,
  \ldots, {\nn}-1$ span the cone of functions $r \in \cCxy$ given by piecewise
  affine functions with kinks only possible at $\zni$ and $\dom r =
  [\xni[n,1], \xni[n,{\nn}]]$.  Correspondingly, applying $\eval_{\nn}$ to the above set of
  functions gives the span of $\cCxyn$.
  \qed
\end{proof}

Next we study characterizations of the estimators.  First, we state the
result for the unconstrained estimator.  The characterizations are derived
from the previous proposition about the boundary elements of the cones we
minimize over together with the following optimality conditions,
as given in Corollary 2.1 of
\cite{Groeneboom:1996cla}.
We use $\ip{\cdot,\cdot}$ to denote the usual Euclidean inner product and, for a differentiable function $f : \RR^m \to \RR$, we use $\nabla f(x)$ to denote the gradient vector at $x$.
\begin{proposition}[Corollary 2.1, \cite{Groeneboom:1996cla}]
  \label{prop:Fenchel}
  Let $\phi:\RR^{m}\to \RR \cup\{\infty\}$ be differentiable and convex.  Let $z_1, \ldots, z_k \in \RR^m$ and let $\mathcal{K}$ be the convex cone generated by $z_{1},\ldots,z_{k}$. Then $\hat{x}\in\mathcal{K}$ is the minimum of $\phi$ over $\mathcal{K}$ if and only if
  \begin{align}
    \ip{z_{i}, \nabla\phi(\hat{x})} \ge 0 &  \mbox{ for } \; 1\le i\le
    k, \label{eq:Fenchel-conditions-ineq} \\
    \ip{z_{i}, \nabla\phi(\hat{x})}
    =0 & \mbox{ }\mbox{ if }\,\hat{\alpha}_{i}>0, \label{eq:Fenchel-conditions-eq}
  \end{align}
  where the nonnegative numbers $\hat{\alpha}_{1},\ldots,\hat{\alpha}_{k}$
  satisfy $\hat{x}=\sum_{i=1}^{k}\hat{\alpha}_i z_{i}.$
\end{proposition}

\begin{proposition}[\cite{Groeneboom:2001jo}, Lemma 2.6]
  \label{prop:GJW-characterization} Define $\Rnk :=
  \sum_{i=1}^k \rni$ and $\Snk := \sum_{i=1}^k \Yni$ for $k \in \Ix$.  Then $\rn = \argmin_{r
    \in \cCn} \phi_n(r)$ if and only if $\Rnk[n,n] = \Snk[n,n]$ and
  \begin{equation*}
    \label{eq:7}
    \begin{split}
      \sum_{k=1}^{j-1} \Rnk (\txni[n,k+1]-\txni[n,k])
      \begin{cases}
        \le \sum_{k=1}^{j-1} \Snk (\txni[n,k+1]  - \txni[n,k])
        & j \in \Ix, j \ge 2 \\
        = \sum_{k=1}^{j-1} \Snk (\txni[n,k+1]  - \txni[n,k])
        & \mbox{ if } \rn \mbox{ has a kink at } \txni[n,j].
      \end{cases}
    \end{split}
  \end{equation*}
  We  define $\rn$ to always have a kink at $\txni[n,n]$.
\end{proposition}

\begin{proof}
  This is proved in \cite{Groeneboom:2001jo}.  The proof follows from the first part of Proposition~\ref{prop:generators-cC} together with Proposition~\ref{prop:Fenchel}.
  \qed
\end{proof}
The inequality in the characterization is reversed from the original lemma in
\cite{Groeneboom:2001jo}, since we are considering concave regression and
\cite{Groeneboom:2001jo} consider convex regression.
Note that
\eqref{eq:Fenchel-conditions-ineq}
and \eqref{eq:Fenchel-conditions-eq}
are equivalent to saying
\begin{equation}
  \label{eq:Fenchel-finite-T}
  \ip{\Delta, \nabla \phi(\hat{x}) } \ge 0
\end{equation}
for all $\Delta \in \cone \lb \lb z_i : 1 \le i \le k \rb \cup \lb -z_i : i
\in I \rb \rb$
where $I := \lb i : \hat \alpha_i > 0 \rb $  is the set of inactive
constraints
and
\begin{equation*}
  \cone \lb y_i : i \in \mc I \rb := \lb \sum_{i \in \mc I} \alpha_i y_i :
  \alpha_i \ge 0 \rb.
\end{equation*}
The cone we are now interested in is $\cCxyn$ which, by Proposition~\ref{prop:concave-generators}, is generated by
\begin{align}
  a_\pm := \pm (\zni[n,j] - x_0)_{j=1}^{\nn}, \label{eq:def:apm}
\end{align}
\begin{align}
  a_i & :=  \lp(\zni[n,j]-\zni)\one_{\{i \ge j\}} \rp_{j=1}^{\nn}
  & & \mbox{ for } i=2,\ldots,k^0, \mbox{ and} \label{eq:def:ai-neg} \\
  a_i & := \lp (\zni- \zni[n,j]) \one_{\{i \le j\}} \rp_{j=1}^{\nn}
  & & \mbox{ for } i=k^0+1,\ldots, {\nn}-1. \label{eq:def:ai-pos}
\end{align}
We now show an analog of the unconstrained characterization, Proposition~\ref{prop:GJW-characterization}, for the constrained case.  For ease of presentation, we assume without any loss of generality that $x_0 = 0$.
\begin{theorem}
  \label{prop:rnxy-characterization}
  Let $x_0=0$.  For $\urnxy \in \cCxyn$, define
  \begin{align}
    \RnLk & := \sum_{j=1}^k \rnixy[n,j], & &  &  \SnLk & :=\sum_{j=1}^k
    \Wni[n,j], & & \mbox{ for } k = 1, \ldots, k^0-1,   \nonumber \\
    \RnRk & := \sum_{j=k}^{\nn} \rnixy[n,j], &  \mbox{ and } & &  \SnRk & :=\sum_{j=k}^{\nn}
    \Wni[n,j], & & \mbox{ for }  k = k^0+1, \ldots {\nn}. \nonumber
  \end{align}
  Then $\urnxy$ is the unique element of $\argmin_{r \in \cCxyn } \phi_n(r)$
  given by
  Proposition~\ref{prop:rnxy-existence-uniqueness}
  if and only if
  \begin{align}
    & \sum_{k=1}^{j-1} \RnLk (\zni[n,k+1] - \zni[n,k] )
    \le \sum_{k=1}^{j-1} \SnLk(\zni[n,k+1] - \zni[n,k])  \quad  \mbox{ for }  2 \le
    j \le k^0, \label{eq:rnxy-charzn-LHS} \\
    & \sum_{k=j+1}^{\nn} \RnRk (\zni[n,k] - \zni[n,k-1])
    \le \sum_{k=j+1}^{\nn} \SnRk ( \zni[n,k]- \zni[n,k-1])
    \quad \mbox{ for }  k^0+1 \le
    j \le {\nn}-1,  \label{eq:rnxy-charzn-RHS}\\
    & \sum_{k=1}^{k^0 -1} \lp \RnLk - \SnLk \rp ( \zni[n,k+1] - \zni[n,k])
    = \sum_{k=k^0+1}^{\nn} \lp  \RnRk - \SnRk \rp (\zni[n,k] -
    \zni[n,k-1] ), \label{eq:rnxy-charzn-connect-sides}
  \end{align}
  where the inequalities in \eqref{eq:rnxy-charzn-LHS} and
  \eqref{eq:rnxy-charzn-RHS} are equalities if $\zni[n,j]$ is a knot of
  $\rnxy$.
\end{theorem}
\begin{proof}
  Let $a_\pm$ and $a_i$, $2 \le i \le {\nn}-1$, be defined as in
  \eqref{eq:def:apm}, \eqref{eq:def:ai-neg}, and
  \eqref{eq:def:ai-pos}. Compute
  \begin{align}
    \nabla_{\nn} \phi_n(\urnxy) :=
    \lp      \rnixy[n,1] - \Wni[n,1] ,
    \cdots,
    (\rnixy[n,k^0] -\Wni[n,k^0]) \one_{\{{\nn}=n \}} ,
    \cdots ,
    \rnixy[n,{\nn}]- \Wni[n,{\nn}]
    \rp'
  \end{align}
  or $\nabla_{\nn} \phi_n(\urnxy) = \lp (\rnixy[n,j] - \Wni[n,j])
  \one_{\{ j \ne k^0 \mbox{ or } {\nn}=n \} }
  \rp_{j=1}^{\nn}$.  By
  Propositions~\ref{prop:concave-generators} and
  \ref{prop:Fenchel} we see that $\urnxy = \argmin_{r \in \cCxyn} \phi_n(r)$ if and only if $\urnxy \in \cCxyn$ and
  \begin{equation}
    \label{eq:13}
    \ip{ (\zni[n,j])_{j=1}^{\nn}, \nabla_{\nn} \phi_n(\urnxy)} = 0,
  \end{equation}
  \begin{align}
    \label{eq:LHS-generator-ineq}
    \ip{ \lp(\zni[n,j]-\zni)\one_{\{i > j\}} \rp_{j=1}^{\nn},
      \nabla_{\nn} \phi_n(\urnxy)} \ge 0 \mbox{ for } i = 2,\ldots, k^0,
  \end{align}
  and
  \begin{align}
    \label{eq:RHS-generator-ineq}
    \ip{\lp (\zni- \zni[n,j]) \one_{\{i < j\}} \rp_{j=1}^{\nn},
      \nabla_{\nn} \phi_n(\urnxy)} \ge 0 \mbox{ for } i = k^0+1,\ldots, {\nn}-1,
  \end{align}
  with equalities in \eqref{eq:LHS-generator-ineq} and
  \eqref{eq:RHS-generator-ineq} if
  $\xni \in S(\rnxy)$.
  From \eqref{eq:LHS-generator-ineq}, for $i=2, \ldots, k^0-1, k^0$, we have
  \begin{equation}
    \label{eq:10}
    0 \le \sum_{j=1}^{\nn} \one_{\lb j  < i \rb} (\zni[n,j] - \zni) (\rnixy[n,j] -
    \Wni[n,j])
    =  \sum_{j=1}^{i-1}  (\zni[n,j] - \zni) (\rnixy[n,j] -
    \Wni[n,j]),
  \end{equation}
  and from \eqref{eq:RHS-generator-ineq} for $i=k^0+1, \ldots, {\nn}-1$, we have
  \begin{equation}
    \label{eq:11}
    0 \le \sum_{j=1}^{\nn} \one_{\{j > i \}} (\zni - \zni[n,j]) (\rnixy[n,j] -
    \Wni[n,j])
    = \sum_{j=i+1}^{\nn} (\zni - \zni[n,j]) (\rnixy[n,j] - \Wni[n,j]),
  \end{equation}
  and from \eqref{eq:13}
  \begin{equation}
    \label{eq:12}
    0 = \sum_{j =1, j \ne k^0}^{\nn} \zni[n,j] (\rnixy[n,j] - \Wni[n,j])
  \end{equation}
  since $\zni[n,k^0]=0$.
  Summing by parts, we see for $i=2, \ldots, k^0$ that \eqref{eq:10} equals
  \begin{equation}
    \label{eq:charzn-sum-by-parts-L}
    \begin{split}
      \sum_{j=1}^{i-1} \sum_{k=j}^{i-1} (\zni[n,k] - \zni[n,k+1] ) (\rnixy[n,j]
      - \Wni[n,j])
      & = \sum_{j=1}^{i-1} \sum_{k=1}^{i-1} \one_{j \le k}
      (\zni[n,k] - \zni[n,k+1] ) (\rnixy[n,j] -\Wni[n,j]) \\
      & = \sum_{k=1}^{i-1} (\zni[n,k] - \zni[n,k+1])
      \sum_{j=1}^{i-1} \one_{ j \le k} (\rnixy[n,j] - \Wni[n,j]) \\
      & = \sum_{k=1}^{i-1} (\zni[n,k] - \zni[n,k+1])
      \lp \RnLk  - \SnLk \rp.
    \end{split}
  \end{equation}
  Similarly, from \eqref{eq:11}, for $i = k^0+1, \ldots, {\nn}-1,$ we see that
  \begin{equation}
    \label{eq:charzn-sum-by-parts-R}
    \begin{split}
      \sum_{j=i+1}^{\nn} \sum_{k= i+1}^j (\zni[n,k-1] - \zni[n,k] )  \lp
      \rnixy[n,j] - \Wni[n,j] \rp
      & = \sum_{j=i+1}^{\nn} \sum_{k= i+1}^{\nn} \one_{k \le j} ( \zni[n,k-1] -
      \zni[n,k] ) (\rnixy[n,j] - \Wni[n,j] ) \\
      & = \sum_{k=i+1}^{\nn} (\zni[n,k-1] - \zni[n,k]) \sum_{j=i+1}^{\nn} \one_{k \le
        j} (\rnixy[n,j] - \Wni[n,j]) \\
      & = \sum_{k=i+1}^{\nn} (\zni[n,k-1]-\zni[n,k]) (\RnRk - \SnRk).
    \end{split}
  \end{equation}
  To finish, we use the same calculations once more. From \eqref{eq:12},
  since $\zni[n,k^0]=0$, we see that
  \begin{equation}
    \label{eq:charzn-equality-two-sided}
    \sum_{j=1}^{k^0-1} (\zni[n,j] - \zni[n,k^0]) (\rnixy[n,j] - \Wni[n,j])
    = \sum_{j=k^0+1}^{\nn} (\zni[n,k^0] - \zni[n,j]) (\rnixy[n,j] - \Wni[n,j]).
  \end{equation}
  Identifying the left- and right-hand sides of
  \eqref{eq:charzn-equality-two-sided} with the right-hand sides of
  \eqref{eq:10} and \eqref{eq:11} (with $i = k^0$ in both cases), and using
  \eqref{eq:charzn-sum-by-parts-L} and \eqref{eq:charzn-sum-by-parts-R}, we
  see
  \begin{equation*}
    \sum_{k=1}^{k^0-1} (\zni[n,k] - \zni[n,k+1]) ( \RnLk - \SnLk)
    = \sum_{k=k^0+1}^{\nn} (\zni[n,k-1] - \zni[n,k]) (\RnRk - \SnRk).
  \end{equation*}
  This completes the proof. \qed
\end{proof}

\section{Limit process for constrained concave regression}
\label{sec:regression-LSEs-limit}

Now we consider an asymptotic version of this problem.
Let
\begin{equation}
  \label{eq:20}
  dX(t) =  - 12t^2 dt + dW(t)
\end{equation}
where $W$ is a standard two-sided Brownian motion started from $0$.  This serves as a canonical/limiting/white noise version of a concave regression problem (with canonical regression function $r_0(t) = -12t^2$, where the constant $12$ is not important).  As has been seen in past work (e.g., Theorem~\ref{thm:GJWb-regression-asymptotics} in Appendix~\ref{sec:appendix}) and as will be seen below in Theorem~\ref{thm:constrained-regression-asymptotics-sec1}, the white noise problem is important because it yields the limit distribution of (finite sample) estimators.
On a compact interval $[-c,c]$ one can define a least-squares objective function
\begin{equation}
  \label{eq:defn:objective-function}
  \phi_c(r) = \inv{2} \int_{-c}^c r(u)^2\, du - \int_{-c}^c r(u) dX(u),
\end{equation}
as in \cite{Groeneboom:2001fp}.
Note that, symbolically replacing $dX$ with $g$, $(r-g)^2 / 2 = r^2/2 - rg + g^2/2$;
we can drop the $g^2/2$ term (which is irrelevant when optimizing over $r$) which explains
why  \eqref{eq:defn:objective-function} is a `least-squares' objective function.
We can now consider minimizing $\phi_c$ over concave functions $r$ satisfying $r(0)=0$.
See the introduction (pages 1622--1623) of \cite{Groeneboom:2001fp} for further explanation and derivation motivating the idea that \eqref{eq:defn:objective-function} serves as a limit version of the objective function \eqref{eq:8}, and that
\eqref{eq:20} serves as an approximation to the (finite sample) observed data.
For $c > 0$ and $k < 0$, let
\begin{equation}
  \label{eq:defn:C0ck}
  \cCck := \lb  r \colon [-c,c] \to \RR \, ; \,
  r \text{ concave},
  r(0)=0, r(\pm c) = k \rb.
\end{equation}
We add the extra constraints $r(\pm c) = k$ to compactify the problem.  These constraints become irrelevant as $c \to \infty$.
We start by showing existence and uniqueness of the minimizer of \eqref{eq:defn:objective-function}.

\smallskip

\begin{proposition}
  \label{prop:cpct-existence-uniqueness}
  Let $k < 0$ %
  and $\phi_c$ be given by \eqref{eq:defn:objective-function}.
  For Lebesgue-almost every $c > 0$, $\argmin_{r \in \cCck} \phi_c(r)$ exists and is unique with probability $1$.
\end{proposition}

\begin{proof}
  Let $r \in \cCck$.  Note that if $M := \max r \to \infty$, then by concavity of $r$, $r > M / 2$ on some interval of length at least $c/4$ for $M$ large enough.  Then the first term in \eqref{eq:defn:objective-function} is of order $M^2$ whereas the second term is of order $M$, so the objective function value goes to $\infty.$ We can thus almost surely restrict attention to only functions bounded above by some fixed value $M$.

  Now consider the class
  \begin{equation*}
    {\cal C}_{c,k,M}^\circ
    := \lb r : [-c,c]\to \RR,
    r \text{ concave},
    r(0)=0, M \ge r, r(\pm c) \ge k \rb.
  \end{equation*}
  This class is closed under pointwise convergence because limits of concave functions are concave and the limit of a uniformly bounded function is uniformly bounded.
  It is thus a closed subset of a set of functions compact (by Tychonoff's theorem) under pointwise convergence, so is compact, and by the Lebesgue bounded convergence theorem, $\phi_c$ is continuous with respect to pointwise convergence.
  Thus $\phi_c$ attains a minimum on ${\cal C}_{c,k,M}^\circ$.  We now show that the minimum satisfies the constraints $r(\pm c) = k$.  Assume, to the contrary, that $r(c) > k$.
  Let $(y)_- := \min(y,0)$ for $y \in \RR$.
  Let $\Delta
  = \eta ( c- \delta - \cdot)_-$
  for some $\delta,\eta > 0$, where $\delta \eta = g := r(c)-k$, so that $r(c)+\Delta(c)=k$.  We will show that $\phi_c(r + \Delta) < \phi_c(r)$.  Let $\tilde X(u) := X(u)-X(c)$.  Then
  \begin{align*}
    - \int \Delta  d X
    =  - \int \Delta d\tilde X
    =- (\tilde{X} \Delta)(c-\delta, c] + \int \tilde{X} d\Delta
    = \int \tilde{X} d\Delta
  \end{align*}
  since $\Delta(c-\delta)=0$ and $\tilde{X}(c)=0$.
  Here, we let $g(a,b] := g(b)-g(a)$ for a function $g$ and $a < b$.
  The previous display equals
  \begin{align}
    - \eta \int_{c-\delta}^c (W(u)-W(c))du
    + \eta \int_{c-\delta}^c  (4u^3 - 4c^3) du. \label{eq:existencepf-intdX}
  \end{align}
  There exists a sequence of $\delta$'s converging to $0$ such that the first term in \eqref{eq:existencepf-intdX} is $0$ because integrated Brownian motion started from $0$ crosses $0$ an infinite number of times near $0$, almost surely.
  The second term in \eqref{eq:existencepf-intdX} equals, to first order approximation (as $\delta \searrow 0$),
  \begin{align}
    -    \eta  k \int_{c-\delta}^c (u-c)du
    =  \eta  \delta^2 k / 2
    =  g \delta k /2.
    \label{eq:existencepf-2ndterm}
  \end{align}
  On the other hand, the first order term in $\int_{c-\delta}^c ((r+\Delta)^2 - r^2) d\lambda /2$ is
  $\int r \Delta d\lambda$ which equals, to first order,
  \begin{align}
    r(c) \int \Delta d\lambda = -r(c) \int |\Delta| d\lambda
    =  \frac{-r(c)}{2} g \delta
    <  (-k/2) g \delta. \label{eq:existencepf-1stterm}
  \end{align}
  Thus we see that there exists $\delta$
  such that \eqref{eq:existencepf-1stterm}
  plus \eqref{eq:existencepf-2ndterm} is negative, i.e.\
  such that $\phi_c(r + \Delta) - \phi_c(r) < 0$.  Thus the minimum over ${\cal C}_{c,k,M}^\circ$ satisfies $r(\pm c) = k$, and so $\phi_c$ attains a minimum on $\cCck$.

  Uniqueness of the minimum follows from the strict convexity of $\phi_c$ on the convex set $\cCck$: for any $w \in (0,1)$, $r_1, r_2 \in \cCck$,
  \begin{equation*}
    \phi_c( w r_1 + (1-w) r_2)
    = w \phi_c( r_1) + (1-w) \phi_c(r_2)
    - \frac{w (1-w)}{2} \int_{-c}^c (r_1(u)-r_2(u))^2 du,
  \end{equation*}
  where the right side is strictly less than $w \phi_c( r_1) + (1-w) \phi_c(r_2) $ if $\int_{-c}^c (r_1(u)-r_2(u))^2 du > 0$, so that $\phi_c$ is strictly convex.  This completes the proof. \qed
\end{proof}

\bigskip

We now state and prove a characterization of the minimizer of \eqref{eq:defn:objective-function}.  Unlike the unconstrained case, we must explicitly deal with the  knot set $\Sxy$ (defined below) in the statement and the proof of the theorem, because the correct definitions of the processes depends on knots $\TauL$ and $\TauR$ (also defined below).  This complicates definitions, because $\Sxy$ is not necessarily a countable set.  It is known to have Lebesgue measure zero
\citep{Sinaui:1992wl,Groeneboom:2001fp}.
For our next theorems, we let

\begin{theorem}
  \label{thm:characterization-compact}
  Let  $X$ be given by \eqref{eq:20}.
  Fix $c > 0$  and $k < 0$ and let $\rxy \in \cCck$.  Define $\Sxy$ by
  $$(\Sxy)^c:= (\Sxy(\rxy))^c :=
  \lb t \in \RR : (\rxy)''(t)=0 \rb.$$
  For $t \in [-c,c]$, define
  \begin{equation*}
    \TauR := \inf \lp \Sxy \cap [0 ,c) \rp,
    \qquad
    \TauL := \sup \lp \Sxy \cap (-c,0 ] \rp,
  \end{equation*}
  \begin{align}
    \XR(t) &:= \int_{\TauR}^t dX,
    && \XL(t) := \int^{\TauL}_t dX, \label{eq:defn:XR-XL} \\
    \YR(t) &:= \int_{\TauR}^t \XR(u)du,
    && \YL(t) := \int_t^{\TauL} \XL(u)du. \label{eq:defn:YR-YL}
  \end{align}
  Let $\HR$ be the primitive of the primitive of $\rxy$ such that $\HR(c)=\YR(c)$ and $\HR(\TauR)=\YR(\TauR)$.  Let $\HL$ be the primitive of the primitive of $\rxy$ such that $\HL(-c)=\YL(-c)$ and $\HL(\TauL)=\YL(\TauL)$.  Then for Lebesgue-almost-every $c>0$, $\rxy = \argmin_{ r \in \cCck} \phi_c(r)$ if and only if the following three conditions hold:
  \begin{enumerate}
  \item \label{item:middle-equality}  $(\HR-\YR)(0) = (\HL-\YL)(0)$,
  \item \label{item:inequalities} for $c \ge t \ge 0,$ $(\HR-\YR)(t) \le 0$ and for $-c \le t \le 0$, $(\HL-\YL)(t) \le 0$,
  \item \label{item:equalities} and
    \begin{equation}
      \label{eq:characterization-compact-equality-integral-condition}
      \int_{[-c,0]} (\HL-\YL) \, d(\rxy)'
      = 0
      = \int_{[0,c]} (\HR-\YR) \, d(\rxy)'.
    \end{equation}
  \end{enumerate}
\end{theorem}
For completeness we give, in Lemma~\ref{lem:integration-by-parts} in the appendix, integration by parts formulas, which we will use in the proof without further reference.  Recall also that, by Theorem 23.1 of \cite{Rockafellar:1970wy}, a finite, concave function on $\RR$ has well-defined right and left derivatives on all of $\RR$.
\begin{proof}
  Notice $\TauL$ and $\TauR$ are well defined and finite because $\rxy(0)=0$ and $\rxy(\pm c) = k < 0$, so $\rxy$ cannot be affine.

  {\bf Sufficiency:} %
  \label{page:suffiency-proof-compact-characterization} Assume Conditions \ref{item:middle-equality}, \ref{item:inequalities}, and \ref{item:equalities} hold for $\rxy$.  Since $\HR$ is twice differentiable,  if $\HR(c)=\YR(c)$, $\HR'(c)=\YR'(c)$, and $\HR \le \YR$ then there exists a `one-sided parabolic tangent' \cite{Groeneboom:2001fp} to $\YR$ at $c$.  Because $W$ is of infinite variation, for Lebesgue almost all $c > 0$, $\YR$ cannot have such a one-sided parabolic tangent, so we can thus assume that $\HR'(c)=\FR(c) > \XR(c)=\YR'(c)$.  Note we can rule out $\FR(c) < \XR(c)$ because then on an interval $[c-\delta, c]$, $\delta>0$, we would have $\YR < \HR$.
  Similarly, we can assume
  $-\HL'(-c)= \FL(-c) > \XL(-c) = -\YL'(-c)$.

  Note that for functions $q$ and $r$,
  \begin{equation*}
    q^2 - r^2 = (q-r)^2 + 2r(q-r) \ge 2r(q-r).
  \end{equation*}
  Thus for any $q \in \cCck$ which is not Lebesgue-a.e.\ identical to $\rxy$,
  \begin{align*}
    \phi_c(q)-\phi_c(\rxy)
    &  > \int_{-c}^c ( q - \rxy) \lp \rxy d\lambda - dX \rp \\
    & = -\int_{-c}^0 (q- \rxy) \, d( \FL-\XL)
    + \int_0^c (q-\rxy) \, d(\FR-\XR),
  \end{align*}
  from \eqref{eq:objective-fn-derivative-2-sided}.
  The previous display equals
  \begin{align}
    \MoveEqLeft - \ls ((q-\rxy)(\FL-\XL))(-c,0]
    - \int_{-c}^0 (\FL-\XL) \, d(q-\rxy) \rs \nonumber \\
    & \quad + ((q-\rxy)(\FR-\XR))(0,c]
    - \int_0^c (\FR-\XR) \, d(q-\rxy)  \nonumber \\
    & =
    \int_{-c}^0 (\FL-\XL) \, d(q-\rxy)
    - \int_0^c (\FR-\XR) \, d(q-\rxy) \label{eq:objective-fn-derivative-2-sided-2}
  \end{align}
  since $q,\rxy \in \cCck$, and, recalling $-(\FL-\XL) = (\HL-\YL)'$, $(\FR-\XR)= (\HL-\YL)'$,  we see the previous display equals
  \begin{align*}
    \MoveEqLeft
    -((\HL-\YL)(q-\rxy)')(-c,0]
    + \int_{(-c,0]} (\HL-\YL)\, d(q-\rxy)' \\
    & \quad - ((\HR-\YR)(q-\rxy)')(0,c]
    + \int_{(0,c]} (\HR-\YR) \, d(q-\rxy)',
  \end{align*}
  and if both $q'(\pm c)$ and $(\rxy)'(\pm c)$ are finite, then by Condition~\ref{item:middle-equality} and since $(\HL-\YL)(-c)=0$, $(\HR-\YR)(c)=0$, the previous display equals
  \begin{equation*}
    \int_{(-c,0]}(\HL-\YL) dq'
    + \int_{(0,c]} (\HR-\YR) dq' \ge 0.
  \end{equation*}
  The final inequality follows by Condition~\ref{item:inequalities} and because $q'$ is nonincreasing ($q$ is concave), so that $q'$ defines a nonpositive measure.

  We now show that we can take both $q'(\pm c)$ and $(\rxy)'(\pm c)$ to be finite, which will complete the proof of sufficiency.  \label{page:end-suffiency-proof-compact-characterization} Recall from the beginning of this sufficiency proof that we may assume that $\YR$ does not have a one-sided parabolic tangent at $c$, and thus that $\FR(c) > \XR(c)$.  Now,
  \begin{equation}
    \label{eq:half-objective-fn}
    k(\XR(c)-\FR(c)) + \int_0^c (\rxy)^2(u)du - \int_0^c \rxy(u) dX(u)
  \end{equation}
  equals $\int_{0}^c (\rxy)'(u) (\XR(u)-\FR(u)) du$. But if $(\rxy)'(u) \to -\infty$ as $u \nearrow c$, then $\int_{0}^{\eta} (\rxy)'(u) (\XR(u)-\FR(u)) du \to \infty$ as $\eta \nearrow c$.  A similar argument holds on $[-c,0]$ to show that if $(\rxy)'(u) \to \infty$ as $u \searrow -c$, then $\int_{\eta}^0 (\rxy)'(u) ( \XL-\FL)(u)\, du \to \infty$ as $\eta \searrow -c$.
  Comparison with e.g.\ the triangle function linearly interpolating between $\rxy(\pm c)=k$ and $\rxy(0)=0$, shows that $\phi_c( \rxy ) < \infty$ and so
  \eqref{eq:half-objective-fn} is also $< \infty$.
  Thus, by contradiction, we see that $(\rxy)'$ (interpreted appropriately as the left or right derivative) is bounded above at $-c$ and below at $c$, and so by concavity is bounded on all of $[-c,c]$.
  Similarly, to see that $q'$ can be assumed finite, notice that if $q'(u) \to -\infty$ as $u \nearrow c$ or $q'(u) \to \infty$ as $u \searrow -c$, then \eqref{eq:objective-fn-derivative-2-sided-2} would be infinite, so we would be finished.  This completes the proof of sufficiency.

  \medskip

  \noindent

  {\bf Necessity:} Assume $\rxy = \argmin_{r \in \cCck} \phi_c(r)$.  We argue by perturbations of $\rxy$ to show the characterization holds.  For a perturbation $\Delta : [-c,c] \to \RR$, we will say that the perturbation is `acceptable for small $\epsilon$' if for all $\epsilon > 0$ small enough, $\rxy + \epsilon \Delta \in \cCck$.  If for all $\epsilon > 0 $ small enough, $\rxy + \epsilon \Delta$ is concave (but may not satisfy the constraints at $ \pm c$), we say `$\Delta$ preserves concavity for small $\epsilon$.'  If $\rxy + \Delta \in \cCck$ or $\rxy + \Delta$ is concave we will say that $\Delta$ is `acceptable' or `preserves concavity,' respectively (in which case the $\epsilon$ is generally explicitly given).  We let $\FR := \HR'$ and $\FL := -\HL'$ (in analogy with $\XR = \YR', \XL = -\YL'$). Note: this means $\HL(t) = \int^{\TauL}_t \FL d \lambda$ and $\HR(t) = \int_{\TauR}^t \FR d\lambda$.  Recall that $\lambda$ is Lebesgue measure.  For a perturbation $\Delta$ that is acceptable for small $\epsilon$,
  \begin{align}
    0 \le
    \lim_{\epsilon \searrow 0} \epsilon^{-1} \lp    \phi_c \lp \rxy + \epsilon \Delta \rp - \phi_c( \rxy) \rp
    & = \int_{-c}^c \Delta (\rxy d\lambda - dX)
    \label{eq:objective-fn-derivative} \\
    & = - \int_{-c}^0 \Delta \, d(F_L - X_L)
    + \int_{0}^c \Delta \, d(F_R-X_R),
    \label{eq:objective-fn-derivative-2-sided}
  \end{align}
  since $\rxy$ minimizes $\phi_c$.  We now show a preliminary result.  For $t \ge 0$ let $\Delta_t(u) := (t - u)_-$ and for $t < 0$ let $\Delta_t(u) := (u-t)_-$, where $(y)_- = \min(y,0)$.  Now fix $ t \ge 0$, $\epsilon > 0$, and %
  $\tau \in \Sxy$. Assume $\tau \ge 0$; the case $\tau \le 0$ is analogous.  Assume further that either $(\rxy)'(\tau+) \ne (\rxy)'(\tau-)$ or $(\rxy)_+''(\tau) \ne 0$, where $(\rxy)_+''(\tau) := \lim_{h \searrow 0} h^{-1} \lp (\rxy)'(\tau+h) - (\rxy)'(\tau+) \rp$ is the second derivative from above.  In the statement ``$(\rxy)_+''(\tau) \ne 0$,'' we allow the possibility $(\rxy)_+''(\tau)$ is undefined.  Notice that there exists a sequence of points $\lb \tau_i \rb \subset \Sxy$, $\tau_i \searrow \TauR$, such that $\tau_i$ satisfies the conditions just described for $\tau$, since $\rxy$ is linear on $[\TauL,\TauR]$ (so has second derivative that is $0$ from below).  Thus either $\TauR \in \Sxy$ or there are $\tau_i \in \Sxy$, $\tau_i \searrow \TauR$, all either having discontinuous derivative or having nonzero second derivative from above.
  Now for $s \ge 0$, define the concave function $r_{\epsilon,s}$ by
  \begin{equation}
    \label{eq:defn:r-eps-s}
    \begin{split}
      r_{\epsilon,s}(u)
      & := \min \lp \rxy(s) + ( (\rxy)'(s+) + \epsilon) (u-s) , \; \rxy(u) - \epsilon \Delta_s(u) \rp\\
      & = \rxy(u) \one_{(-\infty,s-\delta]}(u)
      +  \lp\rxy(s) + ( (\rxy)'(s+) + \epsilon) (u-s) \rp \one_{(s-\delta,s)}(u) \\
      & \qquad +  \lp \rxy(u) - \epsilon \Delta_s(u) \rp \one_{[s,\infty)}(u)
    \end{split}
  \end{equation}
  where the equality holds for some small $\delta \equiv \delta_{\rxy,s,\epsilon} \ge 0$  which solves
  \begin{equation}
    \label{eq:delta-equation}
    \rxy(s-\delta) = \rxy(s) - ((\rxy)'(s+) + \epsilon )\delta.
  \end{equation}
  By our assumptions on $\tau$, we can check that there exist sequences $s_i \searrow \tau$ and $\epsilon_i \searrow 0$ such that
  $\delta_i := s_i - \tau \ge 0$ satisfy $\delta_i = O(\epsilon_i)$ as $i \to \infty$.
  If $(\rxy)'(\tau+) \ne (\rxy)'(\tau-)$ then $-\Delta_\tau$ preserves concavity for small $\epsilon$ and we may take $s_i = \tau$ and $\delta_i = 0$ for any $\epsilon_i > 0 $ small enough, so clearly $\delta_i = O(\epsilon_i)$.
  Similarly, if there exists a sequence $s_i \searrow \tau$ such that $(\rxy)'(s_i+) \ne (\rxy)'(s_i-)$ then we may take $\delta_i = 0$.  If $(\rxy)'$ exists but $(\rxy)_+''(\tau) \ne 0$ then there exists a sequence of points $s_i \searrow \tau$, at which we assume without loss of generality at that $\rxy$ is differentiable, such that
  $\rxy(\tau) - \rxy(s_i) - (\rxy)'(s_i) (\tau-s_i) \le \eta (\tau - s_i)^2$
  for some $\eta < 0$ (by concavity).  (Note: we can also assume without loss of generality, since $\rxy$ is differentiable at $\tau$ and $(\rxy)'$ is monotonic, that $(\rxy)'(s_i) \to (\rxy)'(\tau)$.)  Thus, for a sequence $\gamma_i \ge 1$, by differentiability at $s_i$,
  \begin{equation}
    \label{eq:epsilon-delta-definition}
    \eta (\tau-s_i)^2 \gamma_i = \rxy(\tau) - \rxy(s_i) - (\tau-s_i) (\rxy)'(s_i) = o(\tau-s_i)
    \quad \mbox{ as } s_i \searrow \tau.
  \end{equation}
  Thus, set $\epsilon_i := (-\eta)(\gamma_i)(s_i-\tau)$. By \eqref{eq:epsilon-delta-definition}, $\epsilon_i \searrow 0$, and with $\delta_i = s_i-\tau$, we see $\delta_i = O(\epsilon_i)$ as $i \to \infty$, since $\gamma_i \ge 1$ and $\eta < 0$ is fixed.

  Our first goal is to show
  \begin{equation}
    \label{eq:modified-elbow-perturb}
    \lim_{i \to \infty} \frac{ \phi_c(r_{\epsilon_i}) - \phi_c(\rxy) }{\epsilon_i}
    = \lim_{\epsilon \searrow 0} \frac{ \phi_c(\rxy - \epsilon \Delta_{\tau}) - \phi_c(\rxy) }{\epsilon}.
  \end{equation}
  Note that if $\tau$ is an isolated knot then $r_{\epsilon,\tau} = \rxy - \epsilon \Delta_\tau$ for $\epsilon$ small enough, and also that for $\epsilon$ small enough, $\rxy - \epsilon \Delta_\tau$ is concave (thus for small $\epsilon,$ $- \Delta_\tau$ preserves concavity  although it is not an acceptable perturbation).  If $\tau$ is not an isolated knot then $\rxy - \epsilon \Delta_\tau$ is not concave.  However, $r_{\epsilon_i,s_i}$ is indeed concave and by \eqref{eq:modified-elbow-perturb}, we can %
  use $\tilde{\Delta}_{i}:= \rxy - r_{\epsilon_i,s_i}$ in place of $\epsilon \Delta_\tau$.  Now, notice that
  \begin{equation}
    \label{eq:modified-elbow-perturb-main-side}
    \lim_{i \to \infty} \frac{ \phi_c(r_{\epsilon_i} \one_{[s_i,\infty)}) - \phi_c(\rxy \one_{[s_i,\infty)}) }{\epsilon_i}
    = \lim_{i \to \infty} \frac{ \phi_c(\rxy - \epsilon_i \Delta_{s_i}) - \phi_c(\rxy) }{\epsilon_i}
    = - \int \Delta_\tau (\rxy d\lambda - dX),
  \end{equation}
  by the same calculation as in \eqref{eq:objective-fn-derivative}, since $s_i \to \tau$. %
  Here, for a set $A$, $\one_{A}(u)$ is $1$ if $u \in A$ and $0$ otherwise.
  Thus,
  we will show
  \begin{equation}
    \label{eq:modified-elbow-perturb-negligible-side}
    \lim_{i \to \infty} \frac{ \phi_c(r_{\epsilon_i} \one_{(-\infty,s_i)})
      - \phi_c(\rxy \one_{(-\infty,s_i)}) }{\epsilon_i} = 0,
  \end{equation}
  and then conclude that \eqref{eq:modified-elbow-perturb} holds.  We assume without loss of generality that $(\rxy)'(s_i-) = (\rxy)'(s_i+)$ (since if this does not hold for an infinite subsequence of $\lb s_i \rb$, then we can take the subsequence as our sequence, and then $-\Delta_{s_i}$ preserves concavity, $\delta_i=0$, and \eqref{eq:modified-elbow-perturb-negligible-side} is immediate).  Now %
  $ \phi_c(r_{\epsilon_i} \one_{(-\infty,s_i)}) - \phi_c(\rxy \one_{(-\infty,s_i)}) $ equals
  \begin{equation}
    \label{eq:phi-diff}
    \begin{split}
      \MoveEqLeft
      \inv{2} \int_{\tau}^{s_i} \lp \lp \rxy(s_i) + ( (\rxy)'(s_i+) + \epsilon_i) (u-s_i) \rp^2
      - (\rxy)^2(u) \rp \, du  \\
      & - \int_{\tau}^{s_i}  \lp \rxy(s_i) - \rxy(u) + ( (\rxy)'(s_i+) + \epsilon_i) (u-s_i) \rp
      dX(u).
    \end{split}
  \end{equation}
  The first term in \eqref{eq:phi-diff} equals
  \begin{align*}
    \MoveEqLeft \inv{2} \int_{\tau}^{s_i} \bigg[
    ( \rxy(s_i) + (\rxy)'(s_i-) (u-s_i))^2 + 2\epsilon_i (\rxy(s_i) + (\rxy)'(s_i-)(u-s_i)) (u-s_i) +  \epsilon_i^2(u-s_i)^2 %
    \\
    & %
    - \lp   \lp\rxy(s_i) + \rxy(s_i-) (u-s_i)\rp^2 + o(u-s_i) \lp \rxy(s_i) + \rxy(s_i-) (u-s_i)
    \rp + o(u-s_i)^2 \rp
    \bigg] \, du
  \end{align*}
  and since $\delta_i = O(\epsilon_i)$, the previous display is $O(\epsilon_i^2)$.  Recalling $\tilde \Delta_{i} = \rxy - r_{\epsilon_i,s_i}$ is, on $[\tau,s_i]$, the integrand of the second term in \eqref{eq:phi-diff}, we see that the negative of the second term in \eqref{eq:phi-diff} equals
  \begin{equation}
    \label{eq:error-2nd-term-integ-by-parts}
    (X \tilde{\Delta}_i )(\tau, s_i] - \int_{\tau}^{s_i} X(u) ((\rxy)'(s_i-)du - d(\rxy)(u)).
  \end{equation}
  Recall: for a function $g$ and $a<b$, we let $g(a,b] := g(b)-g(a)$.  Since $\tilde{\Delta}_i(s_i)=0$,  $\delta_i = O(\epsilon_i)$ as $\epsilon_i \searrow 0$,  $\tilde{\Delta}_{i}(s_i-\delta_i) = o(\delta_i)= o(\epsilon_i)$ (recall $(\rxy)'(s_i+)=(\rxy)'(s_i-)$), and $X$ is continuous, \eqref{eq:error-2nd-term-integ-by-parts} is $o(\epsilon_i)$.  Thus both
  terms in \eqref{eq:phi-diff} are $o(\epsilon_i)$ and so we have shown
  \eqref{eq:modified-elbow-perturb-negligible-side}, so by \eqref{eq:modified-elbow-perturb-main-side},
  we have shown \eqref{eq:modified-elbow-perturb}.

  From now on we take $\tau = \TauR$.  In the case where $\TauR \notin \Sxy$ but $\TauR = \lim_{i \to \infty} \tau_i$ with $\tau_i \in \Sxy$, the below arguments go through with $\tau = \tau_i$ and taking the limit of $\tau_i$.

  We now show Condition~\ref{item:middle-equality} holds.  Recall $\Delta_t(u) := (t - u)_-$ for $t \ge 0$ and $\Delta_t(u) := (u-t)_-$ for $t < 0$.  Let $\Delta_{-}(u) := -u$.  Now let
  \begin{equation}
    \label{eq:defn:Delta-condition1-center}
    \Delta_{1,i}(u) :=
    -\frac{c \epsilon_i \Delta_{\TauL}}{ \Delta_{\TauL}(-c)}
    +  \epsilon_i \Delta_{-} (u)
    + \tilde{\Delta}_{i}(u) \frac{ c}{ -\Delta_{s_i}(c)}.
  \end{equation}
  Note that $\Delta_{1,i}$ is an acceptable perturbation with $\Delta_{1,i}(\pm c)=0$ and $\Delta_{1,i}(0)=0$.  Furthermore, as in the proof of \eqref{eq:modified-elbow-perturb}, we can show that
  \begin{equation}
    \label{eq:Delta1-derivative}
    0 \le
    \lim_{i \to \infty} \frac{\phi_c( \rxy + \Delta_{1,i})-\phi_c(\rxy)}{\epsilon_i}
    = \int_{-c}^c \lp - \frac{c}{\Delta_{\TauL}(-c)} \Delta_{\TauL} + \Delta_{-} +
    \Delta_{\TauR} \frac{c}{ \Delta_{\TauR}(c)}\rp
    \, \lp \rxy d\lambda - dX \rp.
  \end{equation}
  Let $\Delta_1$ denote the integrand on the right side of \eqref{eq:Delta1-derivative}.
  By \eqref{eq:objective-fn-derivative-2-sided}, the
  right side of \eqref{eq:Delta1-derivative} equals
  $-\int_{-c}^0 \Delta_1 \, d(F_L-X_L)
  + \int_0^c \Delta_1 \, d(F_R -X_R)$ which equals
  \begin{equation*}
    -((F_L-X_L)\Delta_1) (-c,0]
    + \int_{-c}^0 (F_L-X_L) \, d\Delta_1
    + ((F_R-X_R)\Delta_1)(0,c]
    - \int_{0}^c   (F_R-X_R) \, d\Delta_1
  \end{equation*}
  and, because $\Delta_1(\pm c) = 0$, $\Delta_1(0)=0$, the previous display
  equals
  \begin{equation*}
    \begin{split}
      \MoveEqLeft \int_{-c}^0 (F_L-X_L)\, d\Delta_1
      - \int_0^c (F_R-X_R) \, d\Delta_1
      = (H_L-Y_L)(0) - (H_R-Y_R)(0)
    \end{split}
  \end{equation*}
  since by definition $(H_L-Y_L)(-c)$, $(H_R-Y_R)(c),$
  $(H_R-Y_R)(\TauR)$,
  and $(H_L-Y_L)(\TauL)$ are all $0$.
  This shows that $ (H_L-Y_L)(0) - (H_R-Y_R)(0) \ge 0$.

  The perturbation $\Delta_{1,i}$ is based about $\Delta_{-}$.  Since $\Delta_{-}$ does not satisfy the side constraints at $\pm c$, we modified $\Delta_{-}$ by adding two further perturbations, (constant multiples of) $\Delta_{\TauL}$ and $\Delta_{i}$, to yield $\Delta_{1,i}$.  The perturbation $\Delta_{i}$ is approximately equal to $-\Delta_{s_i}$, but modified so as to preserve concavity, and $-\Delta_{s_i}$ is approximately equal to $-\Delta_{\TauR}$.  A totally symmetric argument allows us to use a perturbation based around $- \Delta_{-}(u)=u$ that is modified by adding (constant multiples of) $\Delta_{\TauR}$ and a perturbation that approximates $- \Delta_{\TauL}$.  This shows that $ (H_L-Y_L)(0) - (H_R-Y_R)(0) \le 0$, and allows us to conclude that Condition~\ref{item:middle-equality} holds.

  Now let
  \begin{equation}
    \label{eq:defn:Delta2}
    \Delta_{2,i}(u) := \epsilon_i \Delta_t(u) + \tilde \Delta_{i}(u) \frac{ \Delta_t(c)}{\Delta_{s_i}(c)}.
  \end{equation}
  Then $\Delta_{2,i}(c)=0$ and $\Delta_{2,i}( t \wedge \TauR) = 0$ where $a \wedge b = \min(a,b)$.  Thus $\Delta_{2,i}$ is an acceptable perturbation for all $i$, and, as above one can check that
  \begin{align*}
    0 \le
    \lim_{i \to \infty} \epsilon_i^{-1} \lp \phi_c(\rxy + \Delta_{2,i} ) - \phi_c(\rxy) \rp
    = \lim_{\epsilon \searrow 0}
    \epsilon^{-1} \lp \phi_c(\rxy + \epsilon \Delta_2 ) - \phi_c(\rxy) \rp
  \end{align*}
  where $\Delta_2:= \Delta_t - \Delta_{\TauR} \Delta_t(c) / \Delta_{\TauR}(c)$.  Thus, from
  \eqref{eq:objective-fn-derivative},
  \begin{align}
    0 \le \int_{0}^c \Delta_2 \, d(\FR - d\XR)
    = ((\FR-\XR)\Delta_2) (t \wedge \TauR, c]
    - \int_0^c (\FR-\XR) \, d\Delta_2. \label{eq:delta2-deriv-1}
  \end{align}
  The term $ ((\FR-\XR)\Delta_2) (t \wedge \TauR, c]$ equals $0$ because $\Delta_2$ is $0$ at both $t \wedge \TauR$ and at $c$.  Since $\int (\FR-\XR) d\Delta_{\TauR} = (\HR-\YR)(\TauR,c]=0$, we see \eqref{eq:delta2-deriv-1} equals
  \begin{align*}
    (\HR-\YR)(t,c] = -(\HR-\YR)(t).
  \end{align*}
  This shows Condition~\ref{item:inequalities} holds for $t \ge 0$.  The argument for $t \le 0 $ is analogous.

  Now let
  \begin{align*}
    \Delta_{3+,i} :=
    \epsilon_i \rxy_{\TauR} + \tilde \Delta_{i} \frac{ \rxy_{\TauR}(c)}{\Delta_{s_i}(c)}
  \end{align*}
  where $\rxy_{\TauR} := \one_{[\TauR,\infty)} ( \rxy - \rxy(\TauR))$ is continuous, concave, and satisfies $\rxy_{\TauR}(\TauR)=0$.
  As above, we can check that
  \begin{align}
    0 \le \lim_{i \to \infty} \epsilon_i^{-1} ( \phi_c( \rxy + \Delta_{3+,i})
    - \phi_c(\rxy))
    =
    \int_{0}^c \Delta_{3+} \, d(\FR-\XR)
    \label{eq:delta3plus-1}
  \end{align}
  where $\Delta_{3+} := \rxy_{\TauR} - \Delta_{\TauR} \rxy_{\TauR}(c)/\Delta_{\TauR}(c)$.
  Then \eqref{eq:delta3plus-1} equals
  \begin{align}
    ((\FR-\XR)\Delta_{3+})(\TauR,c] - \int_{\TauR}^c (\FR-\XR) \, d\Delta_{3+}
    = - \int_{\TauR}^c (\FR-\XR) \, d\Delta_{3+}
    \label{eq:delta3plus-2}
  \end{align}
  since $\Delta_{3+}$ is $0$ at $\TauR$ and $c$. Then, since $(\HR-\YR)$ is $0$ at $\TauR$ and at $c$, \eqref{eq:delta3plus-2} equals
  \begin{align}
    - \int_{\TauR}^c (\FR-\XR) \, d \rxy
    & =   -((\HR-\YR)(\rxy)')(\TauR,c]
    + \int_{(\TauR,c]} (\HR-\YR) d(\rxy)' \nonumber \\
    & =  \int_{(\TauR,c]} (\HR-\YR) d(\rxy)'
    \label{eq:delta3plus-3}
  \end{align}
  Here we used that $(\rxy)'(c)$ is finite, which follows from the same argument used in the proof of sufficiency since we have already shown that Condition~\ref{item:inequalities} holds.
  Thus, we have shown that $ \int_{(\TauR,c]} (\HR-\YR) d(\rxy)' \ge 0$.  To show the reverse inequality, let
  \begin{equation*}
    \Delta_{3-,i} :=
    - \epsilon_i \rxy_{\TauR} + \epsilon_i \Delta_{\TauR} \frac{ \rxy_{\TauR}(c)}{\Delta_{\TauR}(c)}.
  \end{equation*}
  Notice that $\rxy + \Delta_{3-,i}$ is concave by checking its right and left derivatives at $\TauR$:
  \begin{equation}
    \label{eq:right-deriv-final-perturb}
    (\rxy)'(\TauR+) - \epsilon_i (\rxy)'(\TauR+) - \epsilon_i \frac{\rxy_{\TauR}(c)}{\Delta_{\TauR}(c)}
    \le (\rxy)'(\TauR+) \le (\rxy)'(\TauR-)
    = (\rxy + \Delta_{3-,i})'(\TauR-),
  \end{equation}
  since $(\rxy)'(\TauR+) \ge -\rxy_{\TauR}(c) / \Delta_{\TauR}(c)$ by concavity.  By \eqref{eq:right-deriv-final-perturb}, we see that $(\rxy + \Delta_{3-,i})'$ is monotonic in a neighborhood of $\TauR$ and thus is monotonic everywhere.  Thus $\Delta_{3-,i}$ is an acceptable perturbation for all $i$.  Then $\Delta_{3-,i}$ is approximately equal to $- \Delta_{3+,i}$, and replicating the arguments in displays \eqref{eq:delta3plus-1}-\eqref{eq:delta3plus-3} shows that $ \int_{(\TauR,c]} (\HR-\YR) d(\rxy)' \le 0$.
  Thus, we can conclude $ \int_{(\TauR,c]} (\HR-\YR) d(\rxy)' = 0$, and an analogous argument shows $\int_{[-c,\TauL)} (\HL-\YL) \, d(\rxy)' = 0$.  We can extend the domain of integration to include $[0,\TauR]$ or $[\TauL,0]$, respectively, since $(\HR-\YR)(\TauR)=0$ and $(\HL-\YL)(\TauL)=0$ and $(\rxy)'' \equiv 0$ on $(\TauL,\TauR)$ by definition.  This shows Condition~\ref{item:equalities} holds and completes the proof of the necessity of Conditions \ref{item:middle-equality}, \ref{item:inequalities}, and \ref{item:equalities}, and thus completes the proof. \qed
\end{proof}

\medskip

\noindent
For $c^- < 0 < c^+$ let
\begin{equation}
  \label{eq:4}
  \phi_c(r) := \inv{2} \int_{c^-}^{c^+} r(u)^2du - \int_{c^-}^{c^+} r(u)dX(u)
\end{equation}
(slightly modifying the definition given in \eqref{eq:defn:objective-function}).
\begin{corollary}
  \label{cor:uniqueness-cpct-knot-endpoints}
  Let $c^- < 0 < c^+$, $k^\pm < 0$, and $M > 0$ be random variables, and let $\phi_c$ be given by \eqref{eq:4}.  Let $\cC^\circ_{c,k}:= \lb r \in \cCxy : M \ge r, r(c^\pm) \ge k^\pm \rb$.  Let $\rxy \in \cC^\circ_{c,k}$ and let $Y_R,Y_L, H_R,\HL$ be as in Theorem~\ref{thm:characterization-compact}, with $\pm c$ replaced by $c^\pm$.  Assume further that $|(\rxy)'(c^\pm)| < \infty$, where $(\rxy)'$ refers to the right or left derivative, and assume that
  \begin{equation}
    \label{eq:3}
    (H_R - Y_R)'(c^+ ) = 0 = (H_L - Y_L)'(c^-).
  \end{equation}
  Then for $M$ large enough, almost surely $\rxy$ is a minimizer of $\phi_c$  and is thus unique in $\cC^\circ_{c,k}$ on $[c^-,c^+]$.
\end{corollary}
\begin{proof}
  The proof of Proposition~\ref{prop:cpct-existence-uniqueness} shows that, for $M$ large enough, there is a minimizer of $\phi_c$ over $\cC^\circ_{c,k}$ and the minimizer is unique on $[c^-,c^+]$.  (The minimizer does not necessarily satisfy $\rxy(c^\pm) = k^\pm$, since $c^\pm$ are random.)
  The proof that $\rxy$ is indeed such a minimizer follows by a slightly modified version of the sufficiency part of the proof of Theorem~\ref{thm:characterization-compact}.  The equality \eqref{eq:objective-fn-derivative-2-sided-2} follows from \eqref{eq:3} (rather than from $\rxy(c^\pm) = k^\pm$).  Note also that we now assume directly that $|(\rxy)'(c^+)| < \infty$ (since $c^\pm$ are random, we do not know that $Y_L, Y_R$ do not have so-called `one-sided parabolic tangents' at $c^\pm$, respectively).  \qed
\end{proof}

\medskip

\noindent
The previous theorem and corollary are used to prove the next theorem, which gives characterizing conditions for a so-called ``value-constrained invelope process'' on all of $\RR$.  (The term ``invelope process'' originates in \cite{Groeneboom:2001fp}.)  The process on $\RR$ governs the limit distribution of
$\rnxy(x_0)$.

\medskip

\begin{theorem}
  \label{thm:characterization-RR-sec1}
  Let $\rxy \in \cCxy$.  Define
  \begin{equation*}
    (\Sxy)^c :=
    ( \Sxy(\rxy))^c
    := \lb t \in \RR : (\rxyc)''(t)=0 \rb.
  \end{equation*}
  Then define
  \begin{equation*}
    \TauR := \inf \lp \Sxy \cap [0 ,\infty) \rp,
    \qquad
    \TauL := \sup \lp \Sxy \cap (-\infty,0 ] \rp,
  \end{equation*}
  and define $X$, $X_L, X_R, Y_L,$ and $Y_R$ as in
  \eqref{eq:20},
  \eqref{eq:defn:XR-XL}, and
  \eqref{eq:defn:YR-YL}.
  For $t \in \RR$, let
  $H_R(t):= \int_{\tau_R}^t \int_{\tau_R}^u \rxy(v)dv du$, and $H_L(t) := \int_t^{\tau_L}  \int_u^{\tau_L} \rxy(v)dv du$.
  Assume
  \begin{enumerate}
  \item \label{item:middle-equality-RR}  $(\HR-\YR)(0) = (\HL-\YL)(0)$,
  \item \label{item:inequalities-RR} for $t \ge 0,$ $(\HR-\YR)(t) \le 0$ and for $t \le 0$, $(\HL-\YL)(t) \le 0$,
  \item \label{item:equalities-RR} and
    \begin{equation}
      \label{eq:characterization-compact-equality-integral-condition-RR}
      \int_{(-\infty,0]} (\HL-\YL) \, d(\rxy)'
      = 0
      = \int_{[0,\infty)} (\HR-\YR) \, d(\rxy)'.
    \end{equation}
  \end{enumerate}
  Then $\rxy$ is unique.
\end{theorem}
Note that $H_L,H_R$ have different definitions in
Theorem~\ref{thm:characterization-compact}
and in Theorem~\ref{thm:characterization-RR-sec1}.
\begin{proof} %

  We need several lemmas for the proof.  The following lemma connects the $H$-processes to the Gaussian processes about which we can make explicit statements and computations.

  \begin{lemma}
    \label{lem:process-cubic-representation}
    Let $\tau_1,  \tau_2 \in \Sxy$ with $0 < \tau_1 < \tau_2$ be such that $\rxy$ is affine on $[\tau_1, \tau_2]$ and let $t \in [\tau_1, \tau_2]$.  Define, for any function $g$, $\nabla g = g(\tau_2)-g(\tau_1)$, $\bar{g} = (g(\tau_1)+g(\tau_2))/2$,
    $\nabla \tau = \tau_2-\tau_1$, and $\bar \tau = (\tau_1+\tau_2)/2$.
    Then
    \begin{equation*}
      \begin{split}
        \HR(t)
        &=
        \frac{ Y_R(\tau_2) (t-\tau_1) + Y_R(\tau_1) (\tau_2-t)}{\nabla \tau} \\
        & \quad
        - \inv{2} \lp \frac{ \nabla X_R}{\nabla \tau}
        + \frac{4}{ (\nabla \tau)^3} ( \bar{X}_R \nabla \tau -
        \nabla Y_R \rp (t-\tau_1)(\tau_2-t),
      \end{split}
    \end{equation*}
    and so $ H_R(\bar \tau) = \bar {Y}_R - \inv{8} \nabla X_{R} \nabla \tau.$
    Analogous formulas can be stated for the left-side processes.
  \end{lemma}
  \begin{proof}
    The proof follows from the proofs of
    Lemma~2.3 of
    \cite{Groeneboom:2001fp},
    and
    Lemma 8.9 of
    \cite{Doss:2016ux}. \qed
  \end{proof}

  \smallskip

  \noindent
  The previous lemma is used to prove the next lemma,  about the ``knot'' behavior of $\rxy$.  

  \begin{lemma}
    Fix $t > 0$.  Let $\tauplusa(t)$ be the infimum of the points of touch of $Y_R$ and $H_R$ in $[t,\infty)$.  Then for all $\epsilon > 0$, there exists $M$, independent of $t$, such that
    $P( \tauplusa(t) - t > M) < \epsilon$.  An analogous statement can be made for the left-side processes, $t < 0$, and the supremum of the points of touch of $Y_L$ and $H_L$ in $(-\infty,t]$.
  \end{lemma}
  \begin{proof}
    The result follows from Lemma~\ref{lem:process-cubic-representation}, via the analysis used
    in the  the proof of Lemma~8.10 of
    \cite{Doss:2016ux}
    (see also Lemma~2.7 of
    \cite{Groeneboom:2001fp}).
    \qed
  \end{proof}

  \smallskip

  \noindent
  The uniqueness of $\rxy$ follows from showing that if two different processes both satisfy the characterizing conditions of the theorem then they are equal.  One considers the cases where the two processes share (sequences of) knots (converging to infinity) or they do not.  The following lemma handles the former case.

  \smallskip

  \begin{lemma}
    Suppose $G_{R,1}$ and $G_{R,2}$ both satisfy the conditions of Theorem~\ref{thm:characterization-RR-sec1} on $H_R$ and $G_{L,1}$ and $G_{L,2}$ satisfy the theorem conditions for $H_L$ (we do not assume a priori that they have the same value for $\TauR$ or $\TauL$, respectively).  Let $r_1 := G_{R,1}'' \equiv G_{L,1}''$ and
    $r_2 := G_{R,2}'' \equiv G_{L,2}''$.    If
    $(G_{R,i}-Y_{R,i})(s_R) = 0$ and $(G_{L,i}-Y_{L,i})(s_L) = 0$, $i=1,2$,
    where $Y_{R,i}$, $Y_{L,i}$ are defined
    by \eqref{eq:defn:YR-YL} based on the knots of $r_i$, $i=1,2$,
    then  $r_1 = r_2 $ on $[s_L,s_R]$.
  \end{lemma}
  \begin{proof}
    This follows from Corollary~\ref{cor:uniqueness-cpct-knot-endpoints}.  We let $Y_L,Y_R$, $H_L$ and $H_R$ be as defined in Theorem~\ref{thm:characterization-RR-sec1},  we  assume $H_R(s_R)-Y_L(s_R)=0$ and $H_L(s_L)-Y_L(s_L)=0$, and we will show that the conditions of Corollary~\ref{cor:uniqueness-cpct-knot-endpoints} are satisfied by $H_L$ and $H_R$.  This will then show the statement of the lemma.
    Since $H_R(\TauR)-Y_R(\TauR)=0$ by definition, we see that $H_R$ is the primitive of the primitive of $H_R''$ satisfying the constant conditions (at $c^+=s_R$ and at $\TauR$) used to define $H_R$ in Corollary~\ref{cor:uniqueness-cpct-knot-endpoints}.  Furthermore, by Condition~\ref{item:inequalities-RR} and because $(H_R-Y_R)(\TauR)=0$, we see that $(H_R-Y_R)'(\TauR) = 0$.  A similar argument can be made for the left-side processes.  Since $H_R^{(3)}$ is finite on $\RR$, the condition $|H_R^{(3)}(c^\pm)|  < \infty$ is automatically satisfied (for either the left or right third derivative).
    Therefore we have shown that the conditions of
    Corollary~\ref{cor:uniqueness-cpct-knot-endpoints}
    are satisfied.
    We apply this to $G_{R,1}$ and $G_{R,2}$. Let
    \begin{equation*}
      M : = \sup_{x \in [s_L,s_R]} \lb r_1(x), r_2(x) \rb, \,
      k^+ : = \min ( r_1(s_R), r_2(s_R)), \,
      \text{ and } \,
      k^- := \min (r_1(s_L), r_2(s_L)).
    \end{equation*}
    Then $r_i \in \cC^\circ_{c,k}$ for $i=1,2$, and both the $i=1$ and $i=2$ processes satisfy the conditions of Corollary~\ref{cor:uniqueness-cpct-knot-endpoints} by the argument in the previous paragraph, so $r_1=r_2$ on $[s_L,s_R]$ as desired. %
    \qed
  \end{proof}

  \smallskip

  \noindent  For the remainder of the proof of Theorem~\ref{thm:characterization-RR-sec1}, one considers cases where on either the left side, the right side, or both sides, there is no sequence of shared touch points converging to infinity, and deriving a contradiction. The argument follows as in the proof of Theorem~5.2 of \cite{Doss:2016ux}.
  This completes the proof
  of Theorem~\ref{thm:characterization-RR-sec1}. \qed
\end{proof}


\noindent
\begin{remark}
  \label{rem:limit-process-rescaling}
  If $X$ is replaced by $X_{a,\sigma}(t) := \sigma W(t) - 4 a t^3$ for
  constants $a, \sigma > 0$, then the conclusion of
  Theorem~\ref{thm:characterization-RR-sec1} still holds; in this case, we
  denote the process $\rxy$ of the theorem by $\rxy_{a,\sigma}$.
\end{remark}

\begin{remark}
  The knot definitions in Theorem~\ref{thm:characterization-RR-sec1} differ from those in Theorem~5.2 of \cite{Doss:2016ux}, in the context of a mode constraint.  Condition~(iii) of Theorem~5.2 of \cite{Doss:2016ux} (which is analogous to Condition~\ref{item:equalities} of Theorem~\ref{thm:characterization-compact}) is based on knots $\tau_+^0 $ and $\tau_-^0$, one (but almost surely not both) of which may be $0$.  These knots are potentially distinct from the knots $\TauL$ and $\TauR$ in that setup, where $\TauL,\TauR$ can never be $0$.
  In the height-constrained problem we consider in this paper, there is only one pair of knots, $\TauL,\TauR$, and they may be $0$; if one is $0$ then both are $0$.
\end{remark}

\section{Asymptotics}
\label{sec:asymptotics}

We can now study the asymptotic behavior of $\rnxy$.  To do so, we will make
the following assumptions on the design.
\begin{assumption}
  \label{assm:design-density-1}
  The design points $\xni \in [0,1]$ satisfy $c/n \le \xni[n,i+1]-\xni \le C/n$,
  $i=1,\ldots,n-1$ for some $0 < c < C < \infty$.  
\end{assumption}
\begin{assumption}
  \label{assm:local-uniform-design-2}
  For $0 \le x \le 1$, let $F_n(x) := n^{-1} \sum_{i=1}^n  \one_{[0,x]}(\xni)$.  There exists $\delta > 0$ such that $\sup_{x : |x-x_0| \le \delta} |F_n(x) - x| = o(n^{-1/5})$.
\end{assumption}

\medskip

\begin{theorem}
  \label{thm:constrained-regression-asymptotics-sec1}
  Suppose that the regression model \eqref{eq:regression-augmented} holds
  where $r_0$ is concave with $r_0(x_0)=y_0$,
  suppose that $\epsilon_{n,1},\ldots,\epsilon_{n,n}$ are i.i.d.\ with
  $E^{\epsilon_{n,1}^2 t} < \infty$ for some $t > 0$,
  that $r_0$ is twice continuously differentiable in a neighborhood of $x_0$, and that $r_0''(x_0)<0$.  Let Assumptions \ref{assm:design-density-1} and \ref{assm:local-uniform-design-2} hold.  Let $a = |r_0''(x_0)|/24$, and $\sigma^2 = \Var \epsilon_{n,i}$.  Let $\rnxy = \argmin_r \phi_n(r)$ where the argmin is over 
  concave functions $r$ such that $r(x_0)=y_0$. 
  Then
  \begin{equation*}
    n^{2/5}
    ( \rnxy(x_0 + t n^{-1/5}) - r_0(x_0) - r_0'(x_0) t n^{-1/5}) \to_d
    \rxy_{a,\sigma}(t)
  \end{equation*}
  in $L^p[-K,K]$ for all $K > 0$,
  where
  $\rxy_{a,\sigma}$ is given in Theorem~\ref{thm:characterization-RR-sec1} (and Remark~\ref{rem:limit-process-rescaling}).
\end{theorem}

\begin{remark}
  We suspect asymptotic distributions and the Wilks phenomenon for $2 \log
  \lambda_n(y_0)$ can be derived under more general conditions than
  Assumption~\ref{assm:local-uniform-design-2}, but this assumption is used
  by \cite{Groeneboom:2001jo} (it is their Assumption 6.1) to derive the
  limit distribution of $\rn(x_0)$, so we rely on it here too and leave
  generalizations for future research.
\end{remark}

\begin{remark}
  We require a sub-Gaussian tail assumption on $\epsilon_{n,i}$ in Theorem~\ref{thm:constrained-regression-asymptotics-sec1}.  In \cite{Brunk:1970wj}, the asymptotic distribution for a monotone regression function estimator is derived under only second moment assumptions for the error variables.  However, for deriving the rates of convergence for concave regression least-squares estimators, \cite[Theorem 4]{Mammen:1991ey} (and then \cite{Groeneboom:2001jo}) assume sub-Gaussian tails on the error variables.  \cite[page 749]{Mammen:1991ey} states, ``We do not believe that this strong condition is really necessary.''  However, in the present paper we have not attempted to weaken this assumption.
\end{remark}

\begin{proof}[of Theorem~\ref{thm:constrained-regression-asymptotics-sec1}]
  We take $x_0 = 0$ for simplicity and take $r_0(0)=0 $ and $r_0'(0)=0$ by the translation discussed in Section~\ref{sec:regression-LSEs-finitesample}.  Let $x_n(t):= tn^{-1/5} = tn^{-1/5} + x_0$ be the ``global'' parameter corresponding to the ``local'' parameter $t \in \RR$.  Then let $\TaunR$ be the smallest nonnegative bend point of $\rnxy$.
  Recall $F_n(x) := n^{-1} \sum_{i=1}^n \one_{[0,x]}(\xni)$.  Then define
  \begin{align*}
    \SS_{n,R}(v) & :=  \int_{[\TaunR,v]} Y(u) \, d F_n(u)  &
    \YnRloc(t) & :=  n^{4/5} \int_{\TaunR}^{x_n(t)}  \SS_{n,R}(u) du \\
    \RR_{n,R}(v)  & :=  \int_{[\TaunR,v]} \rnxy(u) \, dF_n(u)  &
    H_{n,R}(t) & :=  n^{4/5} \int_{\TaunR}^{x_n(t)} \RR_{n,R}(u) du
    + A_{n,R} (t - n^{1/5} \TaunR),
  \end{align*}
  where $Y$ is the function such that $Y(\xni) = \Wni$ (and whose value is $0$ elsewhere), and where
  \begin{equation*}
    A_{n,R} :=
    n^{3/5} \int_{[\TaunR, \infty)} Y(u) - \rnxy(u)    \, dF_n(u).
  \end{equation*}
  Define also
  \begin{equation*}
    \RRnRt(v) := \int_{[\TaunR, v]} \rnxy(u) du,
    \quad
    \HnRt(t) :=  n^{4/5} \int_{[\TaunR, x_n(t)]} \RRnRt(u) du
    + A_{n,R} (t - n^{1/5} \TaunR)
  \end{equation*}
  which we will show to be equivalent to $\RRnR$ and to $\HnR$, respectively.  For brevity, we will make definitions and arguments only for the right-side processes.  Analogous definitions and arguments can be made for the left-side processes.

  By Theorem~\ref{prop:rnxy-characterization}, one can check that
  \begin{align*}
    \HnR(t) - \YnRloc(t)  \le 0
  \end{align*}
  for all $ t \ge 0 $, with equality if $t$ is a knot point of $\rnxy$ (see Lemma 8.18 of \cite{Doss:2016ux} for similar calculations).  Additionally, defining $Y_{n,L}$ and $H_{n,L}$ in an analogous fashion as $\YnRloc$ and $\HnR$, we can check that
  \begin{equation}
    \label{eq:23}
    H_{n,L}(0)-Y_{n,L}(0) = \HnR(0)-\YnRloc(0).
  \end{equation}
  Next, we can check that
  \begin{equation}
    \label{eq:22}
    \sup_{ |t| \le c} |\HnR(t) - \HnRt(t) | = o_p(1)
  \end{equation}
  for any $c > 0$, by Assumption~\ref{assm:local-uniform-design-2} \cite[see page 1696]{Groeneboom:2001jo}.  One can define $\tilde{H}_{n,L}$ and make an analogous statement for $\tilde{H}_{n,L}$ and $H_{n,L}$.
  We can then conclude that
  \begin{equation}
    \label{eq:24}
    \tilde H_{n,L}(0)-Y_{n,L}(0) = \HnRt(0)-\YnRloc(0) + o_p(1),
  \end{equation}
  \begin{align}
    \HnRt(t) - \YnRloc(t)  + o_p(1) & \le 0, \label{eq:HRtilde-ineq} \\
    \tilde H_{n,L}(t)  - Y_{n,L}(t) + o_p(1)  & \le 0 \label{eq:HLtilde-ineq}
  \end{align}
  where the inequalities are equalities for knot points of $\rnxy$.

  Let
  $\SS_n(t):= n^{-1} \sum_{i=1}^n \Wni \one_{\lb \xni \le t \rb}$,
  and let $Y_n(t) := \int_{x_0}^{x_n(t)} ( \SS_n(v) - \SS_n(x_0) ) dv $.  Then, for any $c > 0$, we can then check that $Y_n$ converges weakly to $\sigma \int_0^t W(s)ds - a t^4 = Y_{a,\sigma}(t)$ in the space of continuous functions on $[-c,c]$ with the uniform metric (\cite[(6.12), page 1694]{Groeneboom:2001jo}).  A similarly structured argument shows that
  along certain subsequences of $\lb n \rb_1^\infty$,
  $\YnRloc$ converges to a process $Y_{R, a, \sigma} \equiv \YR$ (which may a priori depend on the subsequence, but eventually is seen not to depend on the subsequence).  The convergence argument for $\YnRloc$ requires more care than that for $Y_n$ because the definition of the former depends on the knot $\TaunR$.  Nonetheless it can be rigorously carried out, in a fashion similar to that of the proofs of Lemmas 8.16 and 8.17 of \cite{Doss:2016ux}.

  Then the remainder of the proof follows as in the proof of Theorem 6.3 of \cite{Groeneboom:2001jo} (see also \cite{Mammen:1991ey}) and of Theorem 5.8 of \cite{Doss:2016ux}.
  By Lemma~\ref{lem:TaunR-Op} below,  $n^{1/5} \TaunR = O_p(1)$, and this allows us to also conclude that $\HnRt$ and its first, second, and third derivatives are all tight in appropriate metric spaces. Then, by Prohorov's theorem, for any subsequence we can find a subsubsequence of $\HnRt$ that converges to a limit process, $H_R$.  The processes  $H_R$ and $Y_R$ can be shown to satisfy $H_R(t) - Y_R(t) \le 0$ for $t \ge 0$ and $ \int_{[0,\infty)} (H_R-Y_R) \, d(H_R)^{(3)} = 0$ by
  \eqref{eq:HRtilde-ineq}. Arguing analogously for left-side processes, we can see that there are limit processes $H_L$ and $Y_L \equiv Y_{L,a,\sigma}$ satisfying $H_L(t) - Y_L(t) \le 0$, and $ \int_{(-\infty,0]} (H_L-Y_L) d(H_L)^{(3)} = 0$  by
  \eqref{eq:HLtilde-ineq}, and by \eqref{eq:24} that $H_L(0)-Y_L(0)= H_R(0)-Y_R(0)$. This shows conditions \ref{item:middle-equality-RR}, \ref{item:inequalities-RR}, and \ref{item:equalities-RR} of Theorem~\ref{thm:characterization-RR-sec1} hold for the processes $H_L$ and $H_R$.  Therefore, the limit processes $H_L$ and $H_R$ are unique, so are identical along all subsequences.  That is, we can conclude that $\tilde{H}_{n,L}$ and $\HnRt$ converge to the unique processes $H_L$ and $H_R$ given by Theorem~\ref{thm:characterization-RR-sec1}.  In particular, we have shown that $n^{2/5} \rnxy(tn^{-1/5}) = (\HnRt)''(t) \equiv (\tilde H_{n,L})''(t)$ converges to $(H_L)'' \equiv (H_R)''$ and so (recalling that $x_0 = 0$ and $r(0)=r'(0)=0$ by assumption) the proof is complete. \qed
\end{proof}

\begin{lemma}
  \label{lem:TaunR-Op}

  Let the assumptions and terminology of Theorem~\ref{thm:constrained-regression-asymptotics-sec1} hold, and let $\TaunR$ be the smallest nonnegative bend point of $\rnxy$.  Then $n^{1/5} \TaunR = O_p(1)$.
\end{lemma}
\begin{proof}
  The proof is by a perturbation argument in the spirit of Theorem~4.3 (and Lemma~4.4) of \cite{Balabdaoui:2009eh} and Proposition~7.3 of \cite{Doss:2016ux} (which in turn are inspired by 
  Lemma~8
  of \cite{Mammen:1991ey}).
  If $0$ is itself a knot of $\rnxy$ then there is nothing to show (because $\TaunR = 0$).  Thus we assume $0$ is not a knot of $\rnxy$.  We will construct a `perturbation' %
  $\Delta \colon \RR \to \RR$
  such that %
  $\ip{\eval_{\nn} \Delta, \nabla_{\nn} \phi_n(\urnxy)} =  \ip{ \eval_{\nn} \Delta, \urnxy - \uYn}\ge 0$ %
  as in \eqref{eq:Fenchel-finite-T},
  where $\uYn := (Y_{n,1}, \ldots, Y_{n,n})'$ (recalling $Y_{n,k^0}:=0$ if $n_0 = n+1$).
  This  implies
  \begin{equation}
    \label{eq:15}
    \int \Delta(u) (\rnxy(u) - r_0(u)) \, dF_n(u) \ge \int \Delta(u) (Y(u) - r_0(u)) \, dF_n(u)
  \end{equation}
  using the notation developed in the proof of Theorem~\ref{thm:constrained-regression-asymptotics-sec1}.  The approach is to find a $\Delta$ such that the quantity on the left side of \eqref{eq:15} is a positive constant times $-\TaunR^4 < 0$, and the quantity on the right side of \eqref{eq:15} is $O_p(n^{-4/5})$. The conclusion then follows.  
  

  Let $\TaunL < 0 < \TaunR$ be the largest negative and smallest positive knots of $\rnxy$, respectively.  
  Assume $\TaunR \le |\TaunL|$, without loss of generality.  
  Let
  \begin{equation*}
    \Delta_1(t) :=   
    t \one_{[\TaunL, 0)} + \lp \frac{ \TaunL}{\TaunR}\rp^3 t \one_{[0, \TaunR]}
  \end{equation*}
  which satisfies $\Delta_1(0)=0$.
  A simple argument
  (see Lemma A.4 of \cite{Dumbgen:2009bw}) shows that even though $\Delta_1$ is discontinuous, the conclusion of \eqref{eq:15} holds, meaning
  \begin{equation}
    \label{eq:21}
    \int \Delta_1(u) (\rnxy(u) - r_0(u)) \, dF_n(u) \ge \int \Delta_1(u) (Y(u) - r_0(u)) \, dF_n(u).
  \end{equation}
  Further,
  \begin{equation}
    \label{eq:26}
    \int \Delta_1(u) u du = 0,
  \end{equation}
  which will later allow us to ignore a term in a Taylor expansion.
  (Note that in \cite{Balabdaoui:2009eh} and 
  \cite{Doss:2016ux}
the perturbation must satisfy  $\int \Delta(x) dx = 0$; in the present case it turns out we do not need this to hold because of the constraint $\rnxy(0)=0$.  On the other hand, we must have $\Delta(0)=0$.)
Now, the empirical process argument used in the proof of Theorem~4 of \cite{Mammen:1991ey} shows that the term on the right of \eqref{eq:21} is $O_p(n^{-4/5})$.  
For the term on the left, we can show that $ \int \Delta_1(u) (\rnxy(u) - r_0(u)) \, dF_n(u) = (1 + o_p(1)) \int \Delta_1(u) (\rnxy(u) - r_0(u)) \, du$ as in \eqref{eq:22}, by Assumption~\ref{assm:local-uniform-design-2}.  Let $D := \rnxy - r_0$.  Since $\rnxy$ is linear on $(\TaunL, \TaunR)$ and $\rnxy(0)=0$,  $\rnxy(u) =  (\rnxy)'(0) u$ for $u \in [\TaunL,\TaunR]$ and so by \eqref{eq:26} and a Taylor expansion of $r_0$ about $0$
(recalling   that $\lambda$ is Lebesgue measure),
  \begin{align}
    \int \Delta_1  D d\lambda
    & =    D(0) \int \Delta_1 d\lambda + D'(0) \int u \Delta_1(u) du
    - \int \Delta_1(u)
    \frac{r_0''(0)}{2} u^2 (1 + o_p(1)) du  \label{eq:no-knot-exp} \\
    & = - \int \Delta_1(u)
    \frac{r_0''(0)}{2} u^2  (1 + o_p(1))du, \label{eq:R1-final}
  \end{align}
  We compute  that $\int u^2 \Delta_1(u) du = (- (\TaunL)^4 + (\TaunL)^3 \TaunR)/4 \le -(\TaunL)^4 /4 < 0$.  Thus we can conclude that the quantity on the left of \eqref{eq:21} equals $-C_{r_0} (\TaunR-\TaunL)^4 (1 + o_p(1))$ for a constant $C_{r_0} > 0$ (since $r_0''(0) < 0$ and $|\TaunL| \ge \TaunR$). 
Thus the proof is complete. \qed
\end{proof}

\subsection{The likelihood ratio statistic}
\label{subsec:LRS-convergence}

Here we present a partial  proof of Conjecture~\ref{conj:Wilks-phenomenon}.  We will break $2 \log \lambda_n$ into two terms, a ``main'' term and a ``remainder'' term. We focus on the main term, which drives the limit distribution (according to simulations), and do not analyze the remainder term (which Conjecture~\ref{conj:Wilks-phenomenon} and simulations would imply to be asymptotically negligible).  To begin, we need to discuss certain rescalings of the processes studied in the previous sections.  For $a , \sigma > 0$, let $X_{a,\sigma}(t) := \sigma W(t) - 4 a t^3$ as in Remark~\ref{rem:limit-process-rescaling}, and, correspondingly, let
\begin{equation}
  \label{eq:16}
  Y_{a,\sigma}(t) :=
  \sigma \int_0^t W(s) ds - at^4
  \stackrel{d}{=} \sigma ( \sigma/a)^{3/5} Y( (a/\sigma)^{2/5} t),
\end{equation}
where the equality in distribution can be checked using the fact that
$W( \alpha \cdot) \alpha^{-1/2} \stackrel{d}{=} W(\cdot)$ for any $\alpha>0$.
Let
$H_{a,\sigma}$ be the
invelope process given by
Theorem~\ref{thm:GJW-process-uniqueness}
based on $Y_{a,\sigma}$,
and let $H^0_{a,\sigma}$ denote either of the  (null hypothesis) invelope processes,
$\HR$ or $\HL$,
given by Theorem~\ref{thm:characterization-RR-sec1}
based on $Y_{a,\sigma}$.
By \eqref{eq:16},
\begin{align*}
  H_{a,\sigma}(t) & \stackrel{d}{=}
  \sigma ( \sigma/a)^{3/5} H_{1,1}( (a/\sigma)^{2/5} t),  
  \quad 
  \text{ and } 
  \quad 
  H^0_{a,\sigma}(t) & \stackrel{d}{=}
  \sigma ( \sigma/a)^{3/5} H_{1,1}^0( (a/\sigma)^{2/5} t).
\end{align*}
Let $\r_{a,\sigma}(t):= (H_{a,\sigma})''(t)$ and $\rxy_{a,\sigma}(t) := (H^0_{a,\sigma})''(t)$ (recall $\HL'' \equiv \HR''$).
Then we have
\begin{align}
  \r_{a,\sigma}(\cdot) & \stackrel{d}{=}
  \sigma^{4/5} a^{1/5} \r( (a/\sigma)^{2/5} \cdot )
  =: \inv{ \gamma_1 \gamma_2^2} \r(\cdot / \gamma_2),
  \label{eq:r-rescaling} \\
  \rxy_{a,\sigma}(\cdot) & \stackrel{d}{=}
  \sigma^{4/5} a^{1/5} \rxy( (a/\sigma)^{2/5} \cdot )
  =: \inv{ \gamma_1 \gamma_2^2} \rxy(\cdot / \gamma_2),
  \label{eq:r0-rescaling}
\end{align}
where we let $\gamma_1 := (a/\sigma)^{3/5} / \sigma$ and
$\gamma_2 := (\sigma/a)^{2/5}$.
This allows us to relate the rescaled processes
$\r_{a,\sigma}$ and $\rxy_{a,\sigma}$ (where $a$ will later depend on $r_0$ and $\sigma^2 = \Var(\epsilon_{n,i})$) to the universal processes
$\r$ and $\rxy$.
For our future use, we note the relationship
\begin{equation}
  \label{eq:17}
  \gamma_1 \gamma_2^{3/2} = \sigma^{-1}.
\end{equation}
We have
\begin{equation}
  \label{eq:wilks-conjecture-pf-1}
  0
  \le 2 \log \lambda_n = 2 (\phi_n( \rnxy) - \phi_n( \rn) )
  = - \sum_{i \in \Ix} \rni^2 - (\rnixy)^2 - (2 \rni \Wni - 2 \rnixy \Wni).
\end{equation}
Now by \eqref{eq:Fenchel-conditions-eq},  $\la \rn, \nabla \phi_n(\rn) \ra = 0$ and
$\la \rnxy, \nabla \phi_n(\rnxy) \ra = 0$, so \eqref{eq:wilks-conjecture-pf-1}
equals
\begin{equation}
  \label{eq:wilks-conjecture-pf-2}
  -\sum_{i\in \Ix} \rni^2 - (\rnixy)^2 - 2( \rni^2 - (\rnixy)^2)
  = \sum_{i \in \Ix} \rni^2 - (\rnixy)^2.
\end{equation}
Now
we expect that away from the constraint, $\rni$ and $\rnixy$ are asymptotically equivalent.  In fact,
we expect that
\eqref{eq:wilks-conjecture-pf-2} can be localized to a sum over indices corresponding to  $O_p(n^{-1/5})$ neighborhoods of $x_0$.
To discuss this, we note that
\eqref{eq:wilks-conjecture-pf-2} can be written as
$n \int_{\RR}  \lp \rn(u)^2 - \rnxy(u)^2 \rp \, dF_n(u)$
(recalling $F_n(x) := n^{-1} \sum_{i=1}^n  \one_{[0,x]}(\xni)$).
Then we let
$x_n(t) := x_0 + n^{-1/5} t$, and can then see that $2 \log \lambda_n$ equals
$\DD_{n,b} + E_{n,b}$ where
\begin{align*}
  \DD_{n,b} & :=    n \int_{x_n(-b)}^{x_n(b)}  \lp \rn(u)^2 - \rnxy(u)^2 \rp  dF_n(u), \text{ and}  \\
  E_{n,b} & := n \int_{\RR \setminus [x_n(-b), x_n(b)]} \lp \rn(u)^2 - \rnxy(u)^2 \rp \,   dF_n(u).
\end{align*}
We conjecture that
$E_{n,b}$ is asymptotically negligible for large enough $n$ and $b$.
As was discussed in the introduction, proving that
$E_{n,b}$ is asymptotically negligible
may be quite
challenging.  A result of this sort was shown fully in
\cite{Doss:2016ux,Doss:2016vq} in the context of a likelihood ratio statistic
for the mode of a log-concave density.  In some contexts where the underlying
shape constraint is one of monotonicity rather than convexity/concavity, the
corresponding problem seems to often be simpler \citep{Banerjee:2001jy,
  Banerjee:2000we,Groeneboom:2015ew}.  It is beyond the scope of the present
paper to show $E_{n,b}$ is negligible; here, we focus on the non-negligible
term $\DD_{n,b}$.

Now, by
Assumption~\ref{assm:local-uniform-design-2}, $\DD_{n,b}$ is
equal to (\cite[page 1695]{Groeneboom:2001jo})
\begin{equation}
  \label{eq:9}
  n \int_{x_n(-b)}^{x_n(b)} \rn(u)^2 - \rnxy(u)^2 \, du
  + o(1)
  = n^{4/5} \int_{-b}^b  \rn( x_n(v))^2 - \rnxy( x_n(v))^2 dv + o(1).
\end{equation}
Let
$a= |r_0''(x_0)|/ 24$.  Let
$\SS_n(t):= n^{-1} \sum_{i=1}^n \Wni \one_{\lb \xni \le t \rb}$,
and let
$Y_n(t) := \int_{x_0}^{x_n(t)} ( \SS_n(v) - \SS_n(x_0) ) dv $.
Then, for any $c > 0$, we can then check that $Y_n$
converges weakly to $\sigma \int_0^t W(s)ds - a t^4 = Y_{a,\sigma}(t)$
in the space of continuous functions on $[-c,c]$ with the uniform metric
(see the proof of Theorem~\ref{thm:constrained-regression-asymptotics-sec1}).
Then,
by (the proofs of)
Theorem~\ref{thm:GJWb-regression-asymptotics} and
by Theorem~\ref{thm:constrained-regression-asymptotics-sec1} (recalling that $r_0(x_0)=0$ and $r_0'(x_0)=0$ by our data translation),
$n^{2/5} \rn(x_n(\cdot))$ converges weakly to
$\r_{a,\sigma}$
and
$ n^{2/5} \rnxy( x_n(\cdot))$ converges weakly to $\rxy_{a, \sigma}$.
Thus, the right side of \eqref{eq:9} converges in distribution to
\begin{align}
  \int_{-b}^b \r_{a,\sigma}^2 - (\rxy_{a,\sigma})^2
  &  = \int_{-b}^b \lp \inv{ \gamma_1 \gamma_2^2} \rp^2
  \lp \r \lp \frac{ s }{ \gamma_2} \rp^2  - \rxy  \lp \frac{ s }{ \gamma_2} \rp^2 \rp ds \nonumber  \\
  &  = \gamma_1^{-2} \gamma_2^{-3} \int_{-b/\gamma_2}^{b/\gamma_2}
  ( \r(u)^2 - \rxy(u)^2 ) du
  = \sigma^2 \int_{-b/\gamma_2}^{b / \gamma_2}
  ( \r(u)^2 - \rxy(u)^2 ) du,     \label{eq:wilks-conjecture-pf-limit-integral}
\end{align}
as $n \to \infty,$ by \eqref{eq:r-rescaling} and \eqref{eq:r0-rescaling}, and
recalling that $ \gamma_1^2 \gamma_2^3 = \sigma^{-2}$ by \eqref{eq:17}.  Now if  we let $b \to \infty$ then \eqref{eq:wilks-conjecture-pf-limit-integral} converges to
\begin{equation}
  \label{eq:wilks-limit-3}
  \sigma^2 \int_{-\infty}^\infty   ( \r(u)^2 - \rxy(u)^2 ) du
  =: \DD,
\end{equation}
which does not depend on $r_0$, as desired.  This shows that Conjecture~\ref{conj:Wilks-phenomenon} holds, assuming that $E_{n,b}$ is appropriately negligible.  We thus now state Conjecture~\ref{conj:Wilks-phenomenon} as a theorem under the following assumption on the error term.

\begin{assumption}
  \label{assm:remainder-assm}
  For all small enough $\delta > 0$ there exists $b_\delta > 0$ such that $| E_{n,b_\delta} | \le \delta K$ where $K = O_p(1)$ does not depend on $\delta$.
\end{assumption}

\begin{theorem}
  \label{thm:conjecture-theorem}
  Assume the regression model \eqref{eq:1} holds where $Ee^{t
    \epsilon_{n,i}^2} < \infty$ for some $t > 0$.  Assume ${r}_0$ is concave,
  ${r}_0(x_0) = y_0$, ${r}_0$ is twice continuously differentiable in a
  neighborhood of $x_0$, and ${r}_0''(x_0)<0$.  Let
  Assumption~\ref{assm:design-density-1} and
  \ref{assm:local-uniform-design-2} hold.  Define $2 \log \lambda_n(y_0)$ as
  in \eqref{eq:defn:TLLR}. If
  Assumption~\ref{assm:remainder-assm} holds, then $2 \log \lambda_n(y_0) \to_d \sigma^2
  \DD := \sigma^2 \int_{-\infty}^{\infty} \r(u)^2 - \rxy(u)^2 \, du$.
\end{theorem}
\begin{proof}
  For any $\delta > 0$, for a subsequence of $\lb n \rb_{n=1}^\infty $, there exists a subsubsequence such that along the subsubsequence $E_{n, b_{\delta}} \to_d \delta R$ where $|R| \le K$ almost surely, by Prohorov's theorem and Assumption~\ref{assm:remainder-assm}.  Thus since $\DD_{n,b_{\delta}} \to_d \sigma^2 \int_{-b_\delta / \gamma_2}^{b_\delta / \gamma_2} \r(u)^2 - \rxy(u)^2 \, du =: \sigma^2 \DD_{b_\delta}$ as $n \to \infty$ by \eqref{eq:wilks-conjecture-pf-limit-integral}, we see that $2 \log \lambda_n \to_d \sigma^2 \DD_{b_\delta} + \delta R$ along the subsubsequence.  Taking, say, $\delta = 1$, we see that $2 \log \lambda_n$ has a (tight) limit, which we denote by $\sigma^2 \DD$, along the subsubsequence.  Since $K$ does not depend on $\delta$, we can let $\delta \searrow 0$ so $\delta R \to_p 0$, and since then $b_\delta \nearrow \infty$  we thus see that $\sigma^2 \DD := \sigma^2 \DD_{b_\delta} + \delta R \to \sigma^2 \int_{-\infty}^{\infty} \r(u)^2 - \rxy(u)^2 \, du$ so $\DD = \int_{-\infty}^{\infty} \r(u)^2 - \rxy(u)^2 \, du$.  Thus, along the subsubsequence $2 \log \lambda_n \to_d \sigma^2 \int_{-\infty}^{\infty} \r(u)^2 - \rxy(u)^2 \, du$; since this holds for an arbitrary subsequence, the convergence holds along the original sequence. This completes the proof. \qed
\end{proof}

\section{Simulations}
\label{sec:simulations}

We now use simulation studies to assess our procedures.  First, we give evidence in Figure~\ref{fig:LRasymptotics} that Conjecture~\ref{conj:Wilks-phenomenon} holds.  We simulated from three different true concave regression functions, $-x^2$, $\cos(x)$, and $-\exp(x)$.  We used a fixed design setting, with $n=1000$ points uniformly spaced along an interval.  For $-x^2$ and $\cos(x)$ the intervals were $[-1,1]$.  For $-\exp(x)$ the interval was $[1,3]$.  We used standard normal error terms.
Figure~\ref{fig:LRasymptotics} gives empirical cdfs based on $M=5000$ Monte Carlo replications of the distribution of $2 \log \lambda_{1000}$ for the three regression functions.  The curves are visually indistinguishable, giving evidence in support of Conjecture~\ref{conj:Wilks-phenomenon}.
The curve labeled ``limit'' is based on simulating directly from the distribution of $\DD$.  To do this, we simulated the process $X(t) = W(t) -4t^3$ and computed the limit process $\r$ from
Theorem~\ref{thm:GJW-process-uniqueness}
and $\rxy$ from
Theorem~\ref{thm:characterization-compact}
based on the `data' $X$.
We then computed
$\DD = \int_{\RR} ( \r^2(t) - (\rxy)^2(t) )dt$.
The actual form of the limit is not fundamental to Conjecture~\ref{conj:Wilks-phenomenon}.
However the simulation results reported in  Figure~\ref{fig:LRasymptotics}
appear to indeed show that $\DD$ has this form, since the
``limit'' curve is visually indistinguishable from the other three curves described above.
The final curve is the cdf of a chi-squared distribution with $1$ degree of freedom.  This would be the limit of the likelihood ratio statistic if this were a regular parametric problem, but is distinct from the limit of our likelihood ratio statistic, in this nonparametric problem.

\begin{table}
  \centering
  \begin{tabular}{cccccc}
    $r_0$ &  $x_0$ & $r_0(x_0)$ & $r_0'(x_0)$ & $r_0''(x_0)$ & $d(r_0)$ \\
    \hline
    $-x^2$ &  $0$ & $0$ & $0$ & $-2$ & $1.64$ \\
    $\cos(x)$  & $-\inv{2}$ & $.878$ & $.479$ & $-.878$ & $1.94$ \\
    $-\exp(x)$  & $2$ &  $-7.39$ &  $-7.39$ & $-7.39$ & $1.27$
  \end{tabular}
  \caption{Characteristics of the true concave regression function used in the Monte Carlo simulations.}
  \label{tab:simulation-characteristics}
\end{table}

\begin{table}
  \centering
  \begin{tabular}{cccccc}
    &  $r_0$ & $\alpha = .05$; f & $\alpha = .1$; f  & $\alpha = .05$; r & $\alpha = .1$; r \\
    \hline
    \multirow{3}{*}{$n=1000$}
    & $-x^2$ & $.0466$ & $.0964$ & $.0579$ & $.112$ \\
    & $\cos(x)$ & $.0501$ & $.106$  &  $.0680$ & $.131$  \\
    & $-\exp(x)$ & $.0464$ & $.0964$ & $.0469$ & $.0957$ \\
    \hline
    \multirow{3}{*}{$n=100$}
    & $-x^2$ & $.0527$ & $.108$ & $.0514$ & $.107$ \\
    & $\cos(x)$ & $.0670$ & $.125$  &  $.0650$ & $.126$  \\
    & $-\exp(x)$ & $.0451$ & $.0990$ & $ .0447$ & $.0979$ \\
    \hline
    \multirow{3}{*}{$n=30$}
    & $-x^2$ & $.0591$ & $.116$ & $.0578$ & $.110$ \\
    & $\cos(x)$ & $.0683$ & $.127$  &  $.0694$ & $.127$  \\
    & $-\exp(x)$ & $.0495$ & $.105$ & $ .0455$ & $.0996$
  \end{tabular}
  \caption{Monte Carlo level of the likelihood ratio test procedure for nominal levels $\alpha=.05, .1$.  The column heading ``f'' denotes fixed design, and ``r'' denotes random design.  Results are based on sample sizes of $n=30, 100$, and $1000$, and $M=20000$ Monte Carlo replications.}
  \label{tab:LRT-level}
\end{table}

Thus, with Conjecture~\ref{conj:Wilks-phenomenon} in mind, we implemented our likelihood ratio test for the hypothesis test \eqref{eq:LR-hypothesis-test}, rejecting when $2 \log \lambda_n(y_0) > d_\alpha$, $\alpha \in (0,1)$, where $d_\alpha$ is based on the simulated limit distribution in Figure~\ref{fig:LRasymptotics}.  Specifically, we used the curve based on $-x^2$ with a Gaussian error distribution as the limit distribution for $2 \log \lambda$.  We tested the level under the null hypothesis via Monte Carlo.  Our simulations were based on sample sizes of either $n=30$, $n=100$, or $n=1000$, and $M=20000$ Monte Carlo replications. We used the three $r_0$'s of $-x^2$, $\cos(x)$, and $-\exp(x)$ again on the same intervals listed above.  Two designs were used for each of the $r_0$'s; a fixed design, uniformly spaced, and a random uniform design (although the random design is not covered by our theory).
The reported results are for a standard normal error distribution.
Table~\ref{tab:simulation-characteristics} gives the $x_0$ used for each function's hypothesis test, and $r_0(x_0)$, the true value (which was used for the null hypothesis).  We also report smoothness characteristics of $r_0$ at $x_0$, which could in general affect inference procedures, including the constant $d(r_0) := (24 / \sigma^4 |r_0''(x_0)|)^{1/5}$ with $\sigma = 1$, from \eqref{eq:GJW-asymptotics}.  Table~\ref{tab:LRT-level} gives the simulated levels from the Monte Carlo experiments. The  third and fourth columns give the Monte Carlo level of the test procedure for the two nominal levels of $\alpha=.05$ and $\alpha=.1$, respectively, in the fixed design setting.  The fifth and sixth columns give the results in the random design settings.  The %
results
for $\cos(x)$ were generally the worst, which is perhaps %
attributable to having $x_0$ closer to the edge of the covariate design interval than in the scenarios for the other two regression functions.  Shape constrained estimators suffer near the covariate domain boundary.  We do not present simulation results for coverage of our confidence intervals, since by definition the probability our confidence intervals fail to cover the truth is exactly equal to the level of the corresponding hypothesis test.  We present in
Figure~\ref{fig:CI-bands-plot} a plot of our confidence interval procedure on a single instance of simulated data.

\begin{figure}%
  \centering
  \includegraphics[width=.6\textwidth]{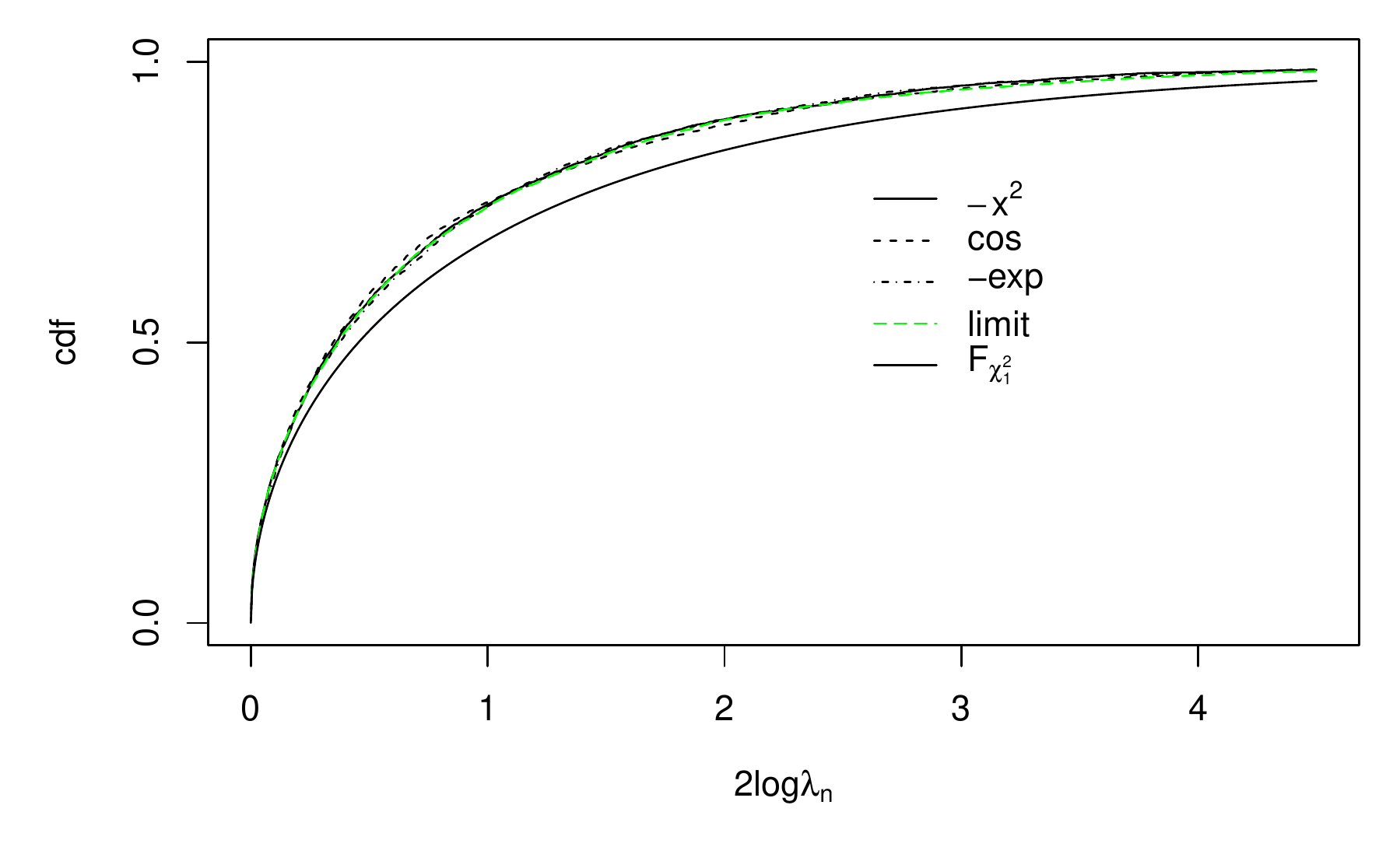}
  \caption{Empirical distributions of $2 \log \lambda_n$ for three different true concave regression functions: $-x^2$, $\cos(x)$, and $-e^x$, all with $n=1000$, $M = 5000$ replications.}
  \label{fig:LRasymptotics}
\end{figure}

\begin{figure}%
  \centering
  \includegraphics[width=1\textwidth,height=.45\textwidth]{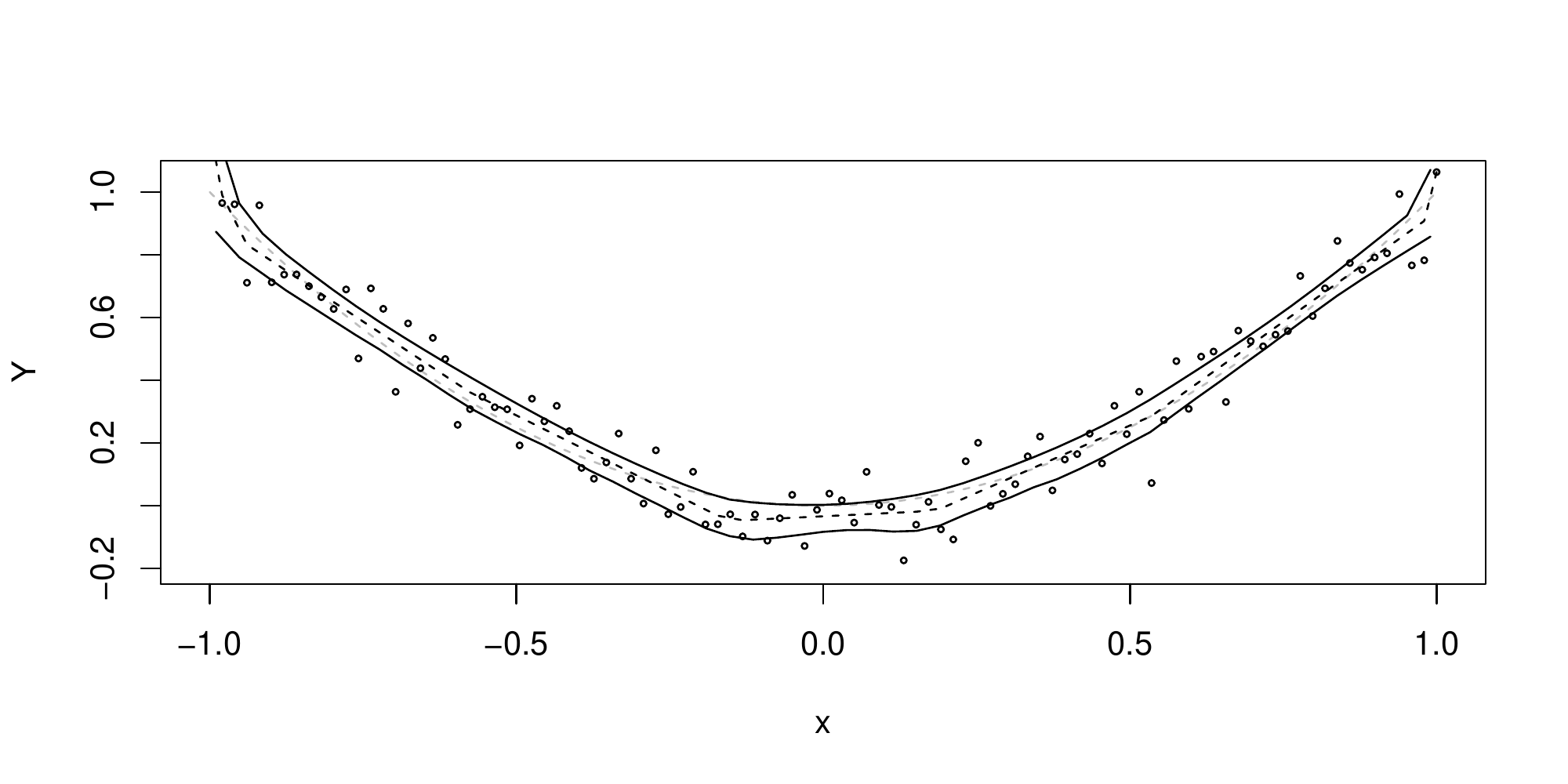}
  \caption{Pointwise confidence interval (solid lines) at each $x \in (-.99,.99)$ via our log likelihood ratio procedure.  Gray dashed line is the true concave regression function $r_0(x)=  x^2$;
    circular points are the simulated data, where $Y_i = r_0(x_i) + \epsilon_i$, $\epsilon_i \iid N(0,.1^2)$, $i=1,\ldots,100$; black dashed line is the ALSE.  }
  \label{fig:CI-bands-plot}
\end{figure}

\section{Conclusions and related problems}
\label{sec:conclusion}

There are several problems related to the concave regression problem discussed in this paper.  We mention two here: the problem of forming tests/CI's
for the value of of a univariate log-concave density, and the problem of forming tests/CI's for the value of a concave/convex regression function with multivariate predictors.

\medskip

{\bf A likelihood ratio for the value of a log-concave density on $\RR$:} In the problem of univariate log-concave density estimation, it is known that the limit distribution of the (univariate) LSE for concave regression \cite{Groeneboom:2001jo} and the (univariate) maximum likelihood estimator for log-concave density estimation \cite{Balabdaoui:2009eh} have the same universal component (they differ in terms of problem-dependent constants).  For studying a height-constrained estimator in the log-concave density problem, the class of interest does not immediately form a convex cone,  but by translation of the log-densities one can arrive at a convex cone.
Consider now $ \tilde X_1, \ldots, \tilde X_n \iid \tilde f_0 = e^{\tilde \vp_0}$ on $\RR$ where $\tilde{\vp}_0 \in \cC$.
Assume $\tilde f_0(x_0) = e^{y_0}$, or $\tilde{\vp_0}(x_0) = y_0$. The nonparametric log likelihood is  $f \mapsto \sum_{i=1}^n \log f(X_i)$.  Following \cite{Silverman:1982vj}, we modify this by a Lagrange term (which allows us to optimize over all concave $\vp$ without regard to the constraint that $\int e^{\vp(z)} dz = 1$).  Optimizing over $\vp  = \log f$, the unconstrained log-concave MLE \cite{Pal:2007eu}  is
$\widehat{\vp}_n := \argmax_{\vp \in \cC} \inv{n} \sum_{i=1}^n
\vp(\tilde{X}_i) - \int e^{\vp(z)} dz.$
As in the concave regression problem, we let $\cCxy := \{ \vp \in \cC \ | \ \vp(x_0)=0 \}.$
We can then consider defining
  $\widehat{\vp}_n^0 := \argmax_{\vp \in \cCxy} \inv{n} \sum_{i=1}^n \vp(\tilde X_i ) + y_0
  - \int e^{y_0 + \vp(z )} dz.$
We can combine $\vvn$ and $\vvna$ to form a likelihood ratio statistic for testing $H_0: \tilde{\vp}(x_0)=y_0$ against $H_1: \tilde{\vp}(x_0) \ne y_0$.  We expect that $\widehat{\vp}_n^0 $ will share features with $\rnxy$ and that the likelihood ratio statistic formed from $\widehat{\vp}_n$ and $\widehat{\vp}_n^0$ will share features with the likelihood ratio statistic \eqref{eq:defn:TLLR} discussed in this paper.  We would expect that it will in fact have the same universal limit distribution $\DD$, independent of nuisance parameters.

\medskip

{\bf A likelihood ratio for the value of a multivariate concave regression function:} Consider the regression model
\begin{equation}
  \label{eq:regression-multivariate}
  \tilde Y_i = \tilde r_0( \tilde x_i) + \epsilon_i,
  \qquad
  i=1,\ldots,n,
\end{equation}
where $\epsilon_i$ are mean $0$ and now $\tilde x_i \in \RR^d$ with $d > 1$.  We are again interested in assuming $\tilde r_0$ is concave and consider estimating it by least-squares, as in \cite{Kuosmanen:2008ks}, \cite{Seijo:2011ko}, and \cite{Lim:2012do}.   We could also consider a constrained estimator as in
\eqref{eq:defn:UC-VC-LSEs}, and form a likelihood ratio statistic for inference about $\tilde r_0(x_0)$ at a fixed point $x_0 \in \RR^d$.
Unfortunately, in the multidimensional case there is no easy analog for Proposition~\ref{prop:concave-generators} describing the generators of the set of concave functions \cite{Johansen:1974ut}, \cite{Bronsten:1978vs}.
Thus, it is unlikely that there are easy analogs of Proposition~\ref{prop:GJW-characterization} and, in the constrained case, Theorem~\ref{prop:rnxy-characterization}.  In fact, while \cite{Seijo:2011ko} and \cite{Lim:2012do} give proofs of consistency of the estimators, pointwise limit distribution results are still unknown.  Again, this is in part because of
the lack of simple generators for the class of multivariate concave functions, so that there is no analog of
Theorem~\ref{thm:characterization-compact}
in the constrained estimator case (or analog of the simpler process studied in
\cite{Groeneboom:2001fp,Groeneboom:2001jo}
in the unconstrained case).
Thus, when $d>1$, making progress in pointwise asymptotics for the estimators and in studying a likelihood ratio statistic  may require new tools or a different approach.

\appendix

\section{Appendix: Technical formulas and other results}
\label{sec:appendix}

Here is a statement of an integration by parts formulas for functions of bounded variation.  See, e.g., page 102 of \cite{Folland1999RealAnalysis} for the definition of bounded variation.
\begin{lemma}[\cite{Folland1999RealAnalysis}]
  \label{lem:integration-by-parts}
  Assume that $F$ and $G$ are of bounded variation on a  set $[a,b]$ where
  $-\infty < a < b < \infty$.
  If at least one of $F$ and $G$ is continuous, then
  \begin{equation*}
    \int_{(a,b]} FdG + \int_{(a,b]} GdF
    = F(b)G(b) - F(a)G(a).
  \end{equation*}
\end{lemma}

\begin{theorem}[\cite{Groeneboom:2001fp}, Theorem 2.1]
  \label{thm:GJW-process-uniqueness}
  Let $\sigma, a > 0$.
  Let $X(t) = \sigma W(t) - 4a t^3$ where $W(t)$ is standard two-sided Brownian motion starting from $0$, and let $Y$ be the integral of $X$ satisfying $Y(0)=0$.  Thus $Y_{a,\sigma}(t) =  \sigma \int_0^t W(s) ds - a t^4$ for $t \in \RR$.  Then, with probability $1$, there exists a uniquely defined random continuous function $H_{a,\sigma}$ satisfying the following:
  \begin{enumerate}[leftmargin=*]
  \item The function $H_{a,\sigma}$ satisfies $H_{a,\sigma}(t) \le Y(t)$ for all $ t\in \RR$.
  \item The function $H_{a,\sigma}$ has a concave second derivative, $\r_{a,\sigma} := H_{a,\sigma}''$.
  \item The function $H_{a,\sigma}$ satisfies $\int_{\RR} ( H_{a,\sigma}(t)-Y_{a,\sigma}(t) ) dH_{a,\sigma}^{(3)}(t)=0$.
  \end{enumerate}
\end{theorem}

\begin{theorem}[\cite{Groeneboom:2001jo}, Theorem 6.3]
  \label{thm:GJWb-regression-asymptotics}
  Suppose that the regression model \eqref{eq:regression-augmented} holds, that $\epsilon_{n,1},\ldots,\epsilon_{n,n}$ are i.i.d.\ with $E^{\epsilon_{n,1}^2 t} < \infty$ for some $t > 0$, that $r_0 \in \cC$, that $r_0''(x_0)<0$, and that $r_0''$ is continuous in a neighborhood of $x_0$.  Let Assumptions \ref{assm:design-density-1} and \ref{assm:local-uniform-design-2} hold.  Let $a := |r_0''(x_0)|/ 24$ and $\sigma^2 := \Var( \epsilon_{n,i})$.  Then
  \begin{equation*}
    n^{2/5}
    ( \rn(x_0 + tn^{-1/5}) - r_0(x_0) - r_0'(x_0) t n^{-1/5} )
    \to_d \r_{a,\sigma}(t)
  \end{equation*}
  in $L^p[-K,K]$ for all $K > 0$.
\end{theorem}

\end{document}